\documentclass[11pt, leqno]{amsart}
\usepackage{amsmath,amssymb,txfonts}
\usepackage{amssymb}
\usepackage{amsxtra}
\usepackage{amsthm, color}
\usepackage{txfonts}
\usepackage{graphicx}
\usepackage{times}
\usepackage{citeref}
\usepackage{tikz}
\usepackage{pgfplots}
\usepackage{tikz-3dplot}

\numberwithin{equation}{section}

\newtheorem{prop}{Proposition}[section]
\newtheorem{theorem}[prop]{Theorem}
\newtheorem{lemma}[prop]{Lemma}
\newtheorem{corollary}[prop]{Corollary}
\newtheorem{remark}[prop]{Remark}
\newtheorem{example}[prop]{Example}
\newtheorem{definition}[prop]{Definition}

\usepackage{tikz}
\usetikzlibrary{calc}

\makeatletter

\def\hyper@x#1,#2\relax{#1}
\def\hyper@y#1,#2\relax{#2}
\def\hyper@coords#1{#1}

\newif\ifhyper@vertical

\def\hyper@computer#1#2{%
	\edef\hyper@toscan{(#1)}
	\tikz@scan@one@point\hyper@coords\hyper@toscan
	\edef\hyper@sx{\the\pgf@x}
	\edef\hyper@sy{\the\pgf@y}
	\edef\hyper@toscan{(#2)}
	\tikz@scan@one@point\hyper@coords\hyper@toscan
	\edef\hyper@ex{\the\pgf@x}
	\edef\hyper@ey{\the\pgf@y}
	\pgfmathsetmacro{\hyper@mx}{(\hyper@ex + \hyper@sx)/2}
	\pgfmathsetmacro{\hyper@my}{(\hyper@ey + \hyper@sy)/2}
	\pgfmathsetmacro{\hyper@dx}{\hyper@ex - \hyper@sx}
	\pgfmathparse{\hyper@dx == 0 ? "\noexpand\hyper@verticaltrue" : "\noexpand\hyper@verticalfalse"}
	\pgfmathresult
	\ifhyper@vertical
	\edef\hyper@cmd{-- (\tikztotarget)}
	\else
	\pgfmathsetmacro{\hyper@dy}{\hyper@ey - \hyper@sy}
	\pgfmathsetmacro{\hyper@t}{\hyper@my/\hyper@dx}
	\pgfmathsetmacro{\hyper@cx}{\hyper@mx + \hyper@t * \hyper@dy}
	\pgfmathsetmacro{\hyper@radius}{veclen(\hyper@cx - \hyper@sx, \hyper@sy)}
	\pgfmathsetmacro{\hyper@sangle}{180 - atan2(\hyper@sy,\hyper@cx-\hyper@sx)}
	\pgfmathsetmacro{\hyper@eangle}{180 - atan2(\hyper@ey,\hyper@cx-\hyper@ex)}
	\edef\hyper@cmd{arc[radius=\hyper@radius pt, start angle=\hyper@sangle, end angle=\hyper@eangle]}
	\fi
}

\def\hyper@disc@computer#1#2{%
	\edef\hyper@toscan{(#1)}
	\tikz@scan@one@point\hyper@coords\hyper@toscan
	\edef\hyper@sx{\the\pgf@x}
	\edef\hyper@sy{\the\pgf@y}
	\edef\hyper@toscan{(#2)}
	\tikz@scan@one@point\hyper@coords\hyper@toscan
	\edef\hyper@ex{\the\pgf@x}
	\edef\hyper@ey{\the\pgf@y}
	\pgfmathsetmacro{\hyper@det}{\hyper@sx * \hyper@ey - \hyper@sy * \hyper@ex}
	\pgfmathparse{\hyper@det == 0 ? "\noexpand\hyper@verticaltrue" : "\noexpand\hyper@verticalfalse"}
	\pgfmathresult
	\ifhyper@vertical
	\edef\hyper@cmd{-- (\tikztotarget)}
	\else
	\pgfmathsetmacro{\hyper@mx}{(\hyper@ex + \hyper@sx)/2}
	\pgfmathsetmacro{\hyper@my}{(\hyper@ey + \hyper@sy)/2}
	\pgfmathsetmacro{\hyper@dx}{\hyper@ex - \hyper@sx}
	\pgfmathsetmacro{\hyper@dy}{\hyper@ey - \hyper@sy}
	\pgfmathsetmacro{\hyper@dradius}{\pgfkeysvalueof{/tikz/hyperbolic disc radius}}
	\pgfmathsetmacro{\hyper@t}{((\hyper@dradius)^2 - \hyper@sx * \hyper@ex - \hyper@sy * \hyper@ey)/(2 * (\hyper@sx * \hyper@ey - \hyper@sy * \hyper@ex))}
	\pgfmathsetmacro{\hyper@radius}{sqrt((\hyper@t)^2 + .25) * veclen(\hyper@dx,\hyper@dy)}
	\pgfmathsetmacro{\hyper@cx}{\hyper@mx + \hyper@t * \hyper@dy}
	\pgfmathsetmacro{\hyper@cy}{\hyper@my - \hyper@t * \hyper@dx}
	\pgfmathsetmacro{\hyper@sangle}{atan2(\hyper@sy-\hyper@cy,\hyper@sx - \hyper@cx)}
	\pgfmathsetmacro{\hyper@eangle}{atan2(\hyper@ey-\hyper@cy,\hyper@ex - \hyper@cx)}
	\pgfmathsetmacro{\hyper@eangle}{\hyper@eangle > \hyper@sangle + 180 ? \hyper@eangle - 360 : \hyper@eangle}
	\edef\hyper@cmd{arc[radius=\hyper@radius pt, start angle=\hyper@sangle, end angle=\hyper@eangle]}
	\fi
}

\def\hyper@plane@tangent#1#2{%
	\edef\hyper@toscan{(#1)}
	\tikz@scan@one@point\hyper@coords\hyper@toscan
	\edef\hyper@sx{\the\pgf@x}
	\edef\hyper@sy{\the\pgf@y}
	\edef\hyper@toscan{(#2)}
	\tikz@scan@one@point\hyper@coords\hyper@toscan
	\edef\hyper@ex{\the\pgf@x}
	\edef\hyper@ey{\the\pgf@y}
	\pgfmathsetmacro{\hyper@ex}{\hyper@ex - \hyper@sx}
	\pgfmathsetmacro{\hyper@ey}{\hyper@ey - \hyper@sy}
	\pgfmathparse{\hyper@ex == 0 ? "\noexpand\hyper@verticaltrue" : "\noexpand\hyper@verticalfalse"}
	\pgfmathresult
	\ifhyper@vertical
	\pgfmathsetmacro{\hyper@d}{\hyper@ey/1cm}
	\pgfmathsetmacro{\hyper@radius}{\hyper@sy * exp(\hyper@d) - \hyper@sy}
	\edef\hyper@cmd{-- ++(0,\hyper@radius pt)}
	\else
	\pgfmathsetmacro{\hyper@d}{\hyper@ex > 0 ? veclen(\hyper@ex,\hyper@ey) : -veclen(\hyper@ex,\hyper@ey)}
	\pgfmathsetmacro{\hyper@radius}{abs(\hyper@sy * \hyper@d / \hyper@ex)}
	\pgfmathsetmacro{\hyper@sangle}{90 + atan(\hyper@ey/\hyper@ex)}
	\pgfkeysgetvalue{/tikz/hyperbolic plane target angle}{\hyper@eangle}
	\ifx\hyper@eangle\pgfutil@empty
	\pgfmathsetmacro{\hyper@d}{\hyper@d/1cm}
	\pgfmathsetmacro{\hyper@ey}{\hyper@ey/1cm}
	\pgfmathsetmacro{\hyper@tanhd}{tanh(\hyper@d)}
	\pgfmathsetmacro{\hyper@eangle}{acos((\hyper@d * \hyper@tanhd - \hyper@ey)/(\hyper@d - \hyper@ey * \hyper@tanhd))}
	\fi
	\edef\hyper@cmd{arc[radius=\hyper@radius pt, start angle=\hyper@sangle, end angle=\hyper@eangle]}
	\fi
}

\tikzset{%
	hyperbolic disc radius/.initial={1cm},
	hyperbolic plane/.style={
		to path={
			\pgfextra{\hyper@computer\tikztostart\tikztotarget}
			\hyper@cmd
		}
	},
	hyperbolic plane tangent/.style={
		to path={
			\pgfextra{\hyper@plane@tangent\tikztostart\tikztotarget}
			\hyper@cmd
		}
	},
	hyperbolic disc/.style={
		to path={
			\pgfextra{\hyper@disc@computer\tikztostart\tikztotarget}
			\hyper@cmd
		}
	},
	hyperbolic plane target angle/.initial={},
}

\makeatother
\begin{document}
	
	\title[\tiny Capacities-masses for Laplace-Beltrami operators on complete Riemannian manifolds]{Capacities-to-masses for Laplace-Beltrami operators on complete Riemannian manifolds}
	
\author{\tiny Xiaoshang Jin}
\address{School of Mathematics and Statistics, Huazhong University of Science and Technology, Wuhan, Hubei 430074, China}
\email{jinxs@hust.edu.cn}
\author{\tiny Jie Xiao}
\address{Department of Mathematics \& Statistics,
	Memorial University, St. John's, NL A1C 5S7, Canada}
\email{jxiao@math.mun.ca}


\subjclass[2020]{}

\begin{abstract}
This paper presents an innovative approach to the capacities-to-masses for the Laplace-Beltrami operators on the complete Riemannian manifolds, unexpectedly solving the open problem posed within \cite[Remark 2.16]{GPS} on the limiting variational capacity.
\end{abstract}

\thanks{The first-named and second-named authors were supported by {NNSF of China
		\# 12201225} and NSERC of Canada \# 202979 respectively.}

	\date{}

\date{}


\maketitle
\tableofcontents
\section{Introduction}\label{s1}

If $(\mathbb M^n,\mathsf{g}=(\mathsf{g}_{ij}))$ is any $2\le n$-dimensional complete Riemannian manifold with not only the metric $\mathsf{g}$ but also the pair of scalar curvature and Laplacian operator pair
$$\left\{{\mathsf{R_g}},\ \Delta_{\mathsf{g}}=\frac{1}{\sqrt{|\text{det}\mathsf{g}|}}\frac{\partial}{\partial x_j}\left(\sqrt{|\text{det}{\mathsf{g}}|}{\mathsf{g}}^{ij}\frac{\partial}{\partial x^i}\right)\right\}\ \ \text{with}\ \ (\mathsf{g}^{ij})=(\mathsf{g}_{ij})^{-1},
$$
then for any {general smooth function $u$ on $\mathbb M^2$ or any positive smooth function $u$ on $\mathbb M^{n\geq 3}$} (simply written as $u\in C^\infty(\mathbb M^n)$) the conformally related metric
$$
\tilde{\mathsf{g}}=\begin{cases} e^{2u}\mathsf{g}&\ \ \text{for}\ \ n=2;\\
u^\frac{4}{n-2}\mathsf{g}&\ \ \text{for}\ \ n>3,
\end{cases}
\text{with scalar curvature}\ \ \mathsf{R}_{\tilde{\mathsf{g}}}
$$
enjoys
not only the partial differential equation (cf. \cite{LP, MS})
$$
\Delta_{\mathsf{g}} u=\begin{cases}
-2^{-1}\big(e^{2u}\mathsf{R}_{\tilde{\mathsf{g}}}-{\mathsf{R_g}}\big)&\text{for}\ \ n=2;\\
\left(\frac{2-n}{4(n-1)}\right)\Big(u^\frac{n+2}{n-2}\mathsf{R}_{\tilde{\mathsf{g}}}-{u}{\mathsf{R_g}}\Big)&\text{for}\ \ n\ge 3.
\end{cases}
$$
{but also the conformal transformation formula
$$
\Delta_{\tilde{\mathsf{g}}} f=\begin{cases}
e^{2u}\Delta_{\mathsf{g}} f\ &\text{for}\ \ n=2;\\
u^{-\frac{n+2}{n-2}}\Big(\Delta_{\mathsf{g}} (fu)-f\Delta_{\mathsf{g}} u\Big) &\text{for}\ \ n\ge 3.
\end{cases}
$$

We are just motived by \cite[4.2]{He} to pay attention to the following $(0,\infty)\ni p$-energy
\begin{equation}
\label{11}
\mathsf{E}(p,u,\mathsf{g})=\int_{\mathbb M^n}|\Delta_{\mathsf{g}} u|^p\,d\upsilon_{\mathsf{g}}.
\end{equation}
\begin{itemize}
	\item
On the one hand, of particular interest is that if
$$
{\mathsf{R_g}}=0\ \text{or}\ {\mathsf{R_{\tilde{g}}}}=0
$$
then
\begin{equation}
\label{12}
\mathsf{E}(p,u,\mathsf{g})=\begin{cases} \int_{\mathbb M^2}\big|2^{-1}e^{2u}\mathsf{R}_{\tilde{\mathsf{g}}}\big|^p e^{-2u}\,d\upsilon_{\tilde{\mathsf{g}}}&\text{for}\ \ n=2;\\
 \int_{\mathbb M^n}\left|\Big(\frac{2-n}{4(n-1)}\Big)\Big(u^\frac{n+2}{n-2}\mathsf{R}_{\tilde{\mathsf{g}}}\Big)\right|^p
u^{\frac{-2n}{n-2}}\,d\upsilon_{\tilde{\mathsf{g}}}&\text{for}\ \ n\ge 3,
\end{cases}
\end{equation}
or
\begin{equation}
\label{13}
\mathsf{E}(p,u,\mathsf{g})=\begin{cases}\int_{\mathbb M^2}\big|2^{-1}\mathsf{R}_g\big|^p\,d\upsilon_{\mathsf{g}}&\text{for}\ \ n=2;\\
 \int_{\mathbb M^n}\left|\Big(\frac{2-n}{4(n-1)}\Big)\mathsf{R}_g\right|^p u^p\,d\upsilon_{\mathsf{g}}&\text{for}\ \ n\ge 3.
\end{cases}
\end{equation}

\item On the other hand, if
$$
\text{$n=2$ or $\Delta_g u=0$ in the case of $n\geq 3$},
$$
then

\begin{equation}
\label{14}
\mathsf{E}(p,f,\tilde{\mathsf{g}})=\begin{cases} \int_{\mathbb M^2}|e^{2u}\Delta_{\mathsf{g}} f|^p e^{-2u} \,d\upsilon_{\mathsf{g}} &\text{for}\ \ n=2;\\
\int_{\mathbb M^n}\left|u^{-\frac{n+2}{n-2}} \Delta_{\mathsf{g}} (fu)\right|^p
u^{\frac{2n}{n-2}}\,d\upsilon_{\mathsf{g}} &\text{for}\ \ n\ge 3,
\end{cases}
\end{equation}
and hence
$$\begin{cases}
\mathsf{E}(1,f,\tilde{\mathsf{g}})=E(1,f,\mathsf{g}) \ &\text{for}\ \ n=2;\\
 \mathsf{E}(\frac{2n}{n+2},f,\tilde{\mathsf{g}})=\mathsf{E}(\frac{2n}{n+2},fu,\mathsf{g})\ &\text{for}\ n\geq 3.
\end{cases}$$
\end{itemize}

In summary, \eqref{11} is related to not only the integral of the scalar curvature but also the conformal invariants under some specific conditions of
$\big\{p, n, (\mathbb M^n,\mathsf{g})\big\}.$ Needless to say, it is also tied to the second-order Adams-Moser-Trudinger-type inequality
at least in the Euclidean spaces; see e.g. \cite{AX}.
}

Accordingly, it is a very natural thing to introduce the following new concept.

    \begin{definition}
    Given $0<p<\infty,$ for any compact subset $K$ of a bounded open subdomain $\Omega$ of $(\mathbb M^n,\mathsf{g})$, the $(2,p)-$capacity of the so-called condenser $(K,\Omega)$ is defined by
    \begin{align}\label{15}
    {\rm cap}_{2,p}(K,\Omega)&=\inf\Bigg\{\int_{\Omega}|\Delta f|^p\,d\upsilon_{\mathsf{g}}:\ f\in C_0^\infty(\Omega)\ \&\ f=1 \ \text{in a neighbourhood of} \ K  \Bigg\}
    \\ & = \inf\Bigg\{\int_{\Omega}|\Delta f|^p\,d\upsilon_{\mathsf{g}}:\ f\in W_0^{2,p}(\Omega)\ \&\ f=1 \ \text{in a neighbourhood of} \ K \Bigg\},\notag
    \end{align}
    where
    $$
    \begin{cases}
    \Delta=\Delta_{\mathsf{g}};\\
    W^{2,p}_0(\Omega)={\overline{C_0^\infty(\Omega)}}^{W^{2,p}(\Omega)};\\
    W^{2,p}(\Omega)=\Big\{f:\ \ \text{$f$'s partial derivatives up to the second order are in $L^p(\Omega)$}\Big\}.
    \end{cases}
    $$

 \end{definition}

    Whenever $K$ and $\Omega$ are smooth domains in $(\mathbb M^n,\mathsf{g})$, we can readily obtain
    \begin{equation}
    \label{16}
    {\rm cap}_{2,p}(K,\Omega)=\inf\Bigg\{\int_{\Omega\setminus K}|\Delta f|^p\,d\upsilon_{\mathsf{g}}:\ f\in \mathfrak{C}^{1,1}(\Omega\setminus K)\Bigg\},
    \end{equation}
    where
    \begin{equation}
    \label{17}
    \begin{cases}\mathfrak{C}^{1,1}(\Omega\setminus K)=\Bigg\{f\in C^{1,1}(\overline{\Omega\setminus K}):\ f\big|_{\partial K}=1\ \ \&\ \ f\big|_{\partial\Omega}=0=|\nabla f|\big|_{\partial (\Omega\setminus K)}\Bigg\};\\
    \nabla\phi=\nabla_g\phi=g^{ij}\frac{\partial\phi}{\partial x_i}\frac{\partial}{\partial x_j}=\text{the gradient operator on}\ C^\infty(\mathbb M^n).
    \end{cases}
    \end{equation}

   {
    Given not only $p\in (1,\infty)$ but also a smooth condenser $(K,\Omega)$ in $(\mathbb M^n,\mathsf{g})$, in order to find the minimum of the $p-$energy
    $$
   \mathsf{E}_p(u)=\int_{\Omega\setminus K}|\Delta u|^p\,d\upsilon_{\mathsf{g}}\ \ \forall\ \  u\in \mathfrak{X}= \mathfrak{C}^{1,1}(\Omega\setminus K)+ W^{2,p}_0(\Omega\setminus K),
    $$
    we consider the following Euler-Lagrange-Dirichlet problem:
   	\begin{equation}
   	\label{18}
   	\begin{cases}
   	\Delta(|\Delta u|^{p-2} \Delta u)=0 \ & \text{in}\ \ \Omega\setminus K;
   	\\ u=1 & \text{on}\ \ \partial K;
   	\\ u=0 & \text{on}\ \ \partial \Omega;
   	\\ |\nabla u|=0 & \text{on}\ \ \partial K\ \&\ \ \partial \Omega,
   	\end{cases}
   	\end{equation}
   thereby establishing the following initial result.

   \begin{theorem}\label{t21} Let not only $p\in (0,\infty)$ but also $(K,\Omega)$ be a smooth condenser in $(\mathbb M^n,\mathsf{g}).$
   	\begin{itemize}
   		\item [\rm (i)]
   If $p\in (1,\infty)$, then not only there exists a unique weak solution $u\in \mathfrak{X}$ to \eqref{18} in the sense that
   \begin{equation}
   	\label{19}
    \int_{\Omega\setminus K}|\Delta u|^{p-2} \Delta u\Delta v\,d\upsilon_{\mathsf{g}}=0\ \ \forall\ \ v\in W^{2,p}_0(\Omega\setminus K)
    \end{equation}
   	but also $u$ is the minimum (said to be the $(2,p)-$potential of $\Omega\setminus K$) of the $p-$energy $\mathsf{E}_p(\cdot)$ such that
    \begin{equation}
   	\label{110}
   	|\Delta u|>0\ \ \text{a.e. on}\ \ \Omega\setminus K\Longrightarrow
   	{\rm cap}_{2,p}(K,\Omega) = \int_{\Omega\setminus K}|\Delta u|^p\,d\upsilon_{\mathsf{g}}.
   	\end{equation}
   \item[\rm (ii)] If $p\in (0,1]$, then there is the following formula
   \begin{equation}
   \label{111}
{\rm cap}_{2,p}(K,\Omega)=\begin{cases} 2{\rm cap}_2(K,\Omega)=2\underset{\mathrm{1}_K\le f\in C_0^\infty(\Omega)}{\inf}\int_{\Omega}|\nabla f|^2\,d\upsilon_{\mathsf{g}} &\ \ \text{for}\ \ p=1;\\
0&\ \ \text{for}\ \ p\in (0,1).
\end{cases}
\end{equation}

\end{itemize}
   \end{theorem}

   \begin{proof} (i) If $p\in (1,\infty)$, then
   	upon fixing a function
   	$$
   	f\in \mathfrak{C}^{1,1}(\Omega\setminus K),
   	$$
   	we have
   $$
   \forall \ u\in\mathfrak{X} \Rightarrow f-u\in W^{2,p}_0(\Omega\setminus K).
   $$
   According to the classic $W^{2,p}$ theory in PDE, there is a positive constant pair $\{C_1,C_2\}$ depending on $p,f$ \& $\Omega\setminus K$ such that
   $$
   \begin{cases}\|f-u\|^p_{2,p}\leq C_1 \mathsf{E}_p(f-u);\\
   \|u\|^p_{2,p}\leq C_1 \mathsf{E}_p(f-u)+\|f\|^p_{2,p}\leq C_2 \big(\mathsf{E}_p(u)+1\big).
   \end{cases}
   $$
   Hence $\mathsf{E}(\cdot)$ is coercive over $\mathfrak{X}.$ Furthermore,
   for any $u_1,u_2\in\mathfrak{X},$  we use $p\in (1,\infty)$ to estimate
   \begin{align*}
   \mathsf{E}_p\big(2^{-1}({u_1+u_2})\big)&=\int_{\Omega\setminus K}\Big|2^{-1}{\Delta u_1}+2^{-1}{\Delta u_2}\Big|^p\,d\upsilon_{\mathsf{g}}\\
   &\leq
   \int_{\Omega\setminus K}\Big({2^{-1}|\Delta u_1|^p}+2^{-1} {|\Delta u_2|^p}\Big)\,d\upsilon_{\mathsf{g}}\\
   &=2^{-1}\big(\mathsf{E}_p(u_1)+\mathsf{E}_p(u_2)\big),
   \end{align*}
   with its inequality becoming an equality iff
   $$\Delta u_1=\Delta u_2\ \ \text{a.e. in}\ \Omega\setminus K.
   $$
   This equality circumstance, along with the given boundary condition, ensures $u_1=u_2$. Hence $\mathsf{E}_p(\cdot)$ is not only strictly convex but also weakly lower semi-continuous over $\mathfrak{X}$ - in other words - there exists a unique
   minimiser $u$ for $\mathsf{E}_p(\cdot)$ which is also the weak solution to \eqref{18}. {Moreover, since $|\Delta u|^{p-2} \Delta u$ is harmonic
   in $\Omega\setminus K,$ we obtain that $|\Delta u|>0$ holds almost everywhere (a.e.) due to the boundary conditions.}

(ii) The formula \eqref{111} will be verified within Theorem \ref{t22} whose case $p=1$ is actually a solution to the open problem posed in \cite[Remark 2.16]{GPS}.
    \end{proof}

   Most importantly, we are led by not only Theorem \ref{t21} but also \eqref{12}-\eqref{13}-\eqref{14} to an intensive case study of
   $$
  {\rm cap}_{2,p}(K,\Omega)\ \&\  {\rm cap}_{2,p}(K)\ \text{connecting to $[1,\infty)]\ni p$-ADM mass $\mathsf{M}_{\mathrm{ADM}_p, \dagger/ \ddagger}(\mathbb M^n,\mathsf{g})$};
   $$
    see the forthcoming consecutive four sections: \S\ref{s3}-\S\ref{s4}-\S\ref{s5}-\S\ref{s6} for the details.
    
\begin{remark}
Given a condenser $(K,\Omega)$ in  $(\mathbb M^n,\mathsf{g})$ and $p>0,$ we can define another $(2,p)-$type capacity by
 \begin{align}
    {\rm C}_{2,p}(K,\Omega)&=\inf\Bigg\{\int_{\Omega}|\Delta f|^p\,d\upsilon_{\mathsf{g}}:\ 1_K\leq f\in C_0^\infty(\Omega) \Bigg\}
    \\ & = \inf\Bigg\{\int_{\Omega}|\Delta f|^p\,d\upsilon_{\mathsf{g}}:\ 1_K\leq f\in W_0^{2,p}(\Omega)\Bigg\}.\notag
    \end{align}
\end{remark}
If $K$ and $\Omega$ are smooth domains, then
 {\small $$
{\rm C}_{2,p}(K,\Omega)=\inf\Bigg\{\int_{\Omega\setminus K}|\Delta f|^p\,d\upsilon_{\mathsf{g}}:\ f\in C^{1,1}(\overline{\Omega\setminus K})\ \&\ f\big|_{\partial K}-1= f\big|_{\partial\Omega}=0=|\nabla f|\big|_{\partial \Omega}\Bigg\}.
$$}
Directly, it follows from the definition that
$${\rm C}_{2,p}(K,\Omega)\leq{\rm cap}_{2,p}(K,\Omega).$$
However, these two capacities are different. In particular, as pointed out by Y. Sun, if $(K,\Omega)$ is such an origin-centered Euclidean round ring that 
$$K=\overline{B(o,\tfrac12)}\subseteq B(o,1)=\Omega\ \ \text{in}\ \ \mathbb{R}^3
$$ 
then 
$$
{\rm C}_{2,2}(K,\Omega)=24\pi<224\pi={\rm cap}_{2,2}(K,\Omega),
$$
whose calculations are omitted here for brevity. In what follows, we will focus on ${\rm cap}_{2,p}(K,\Omega)$ although ${\rm C}_{2,p}(K,\Omega)$ is also interesting.

\section{Capacity-volume estimates}\label{s3}

    In this section, we are driven by \cite[Theorems I.1-II.1]{JX} for the well-known functional/variational capacity
    \begin{equation}
    \label{30}
    [1,\infty)\ni p\mapsto   {\rm cap}_{p}(K,\Omega)=\underset{\mathrm{1}_K\le f\in C_0^\infty(\Omega)}{\inf}\int_{\Omega}|\nabla f|^p\,d\upsilon_{\mathsf{g}}=\underset{\mathrm{1}_K\le f\in W_0^{1,p}(\Omega)}{\inf}\int_{\Omega}|\nabla f|^p\,d\upsilon_{\mathsf{g}}
    \end{equation}
    which relys on
    $$
    \begin{cases}
    W^{1,p}_0(\Omega)={\overline{C_0^\infty(\Omega)}}^{W^{1,p}(\Omega)};\\
    W^{1,p}(\Omega)=\Big\{f:\ \ \text{$f$'s partial derivatives up to the first order are in $L^p(\Omega)$}\Big\},
    \end{cases}
    $$
     to finger out the following new capacity-volume estimation of a smooth condenser in $(\mathbb M^n,\mathsf{g})$.
    \begin{theorem}\label{t41}
	For $n\ge 3$ \& $\sigma_{n-1}$ - the surface-area of the unit sphere $\mathbb S^{n-1}$ of $\mathbb R^n$, suppose that $(K,\Omega)$ is a
     smooth condenser in $(\mathbb M^n,\mathsf{g})$ with volume pair $\big\{|K|,|\Omega|\big\}$. Then
	$$
    \begin{cases} \inf\limits_{K\Subset\Omega}|K|^{-\alpha}{{\rm cap}_{2,p}(K,\Omega)}=0\ \ \forall\ \ \alpha\in\Big(0,1-\frac{2p}{n}\Big)&\ \ \text{as}\ \ 1<p\not=\frac{n}{2};\\
    \inf\limits_{K\Subset \Omega}|K|^{-1}{\exp\left(-\beta \big({\rm cap}_{2,\frac n2}(K,\Omega)\big)^{\frac{2}{2-n}}\right)}=0\ \ \forall\ \
	\beta\in\bigg(0,\frac{n(\sigma_{n-1})^\frac{2}{n-2}}{(n-2)^\frac{n}{2-n}}\bigg)&\ \ \text{as}\ \ p=\frac{n}{2}.
	\end{cases}
	$$
    \end{theorem}

   \begin{remark}
   	\label{r21}
   Here it is perhaps appropriate to emphasize not only that the restrictions of $\{\alpha,\beta\}$ are sharp according to \cite[Theorem 1.2]{AX} for the Euclidean space $(\mathbb R^n,\mathsf{g}_0)$ with
    $$
   \mathsf{g}_0=\Big(\big(\mathsf{g}_0\big)_{ij}\Big)\ \ \text{with}\ \ \big(\mathsf{g}_0\big)_{ij}=\begin{cases} 1& \text{for}\ \ i=j;\\
   0&\text{for}\ \ i\not=j,
   \end{cases}
   $$
   but also Theorem \ref{t41}'s case $p=1$ also holds due to not only \cite[Theorems I.1]{JX} for \eqref{30} but also \eqref{111}'s first identification - namely -
   $$
   {\rm cap}_{2,1}(K,\Omega)=2{\rm cap}_2(K,\Omega).
   $$
   \end{remark}

In order to verify Theorem \ref{t41}, we begin with an investigation of the $\big(2,(1,\infty)\ni p\big)$-capacity of the origin-concentric round ring $\big(\bar{B}(r),B(R)\big)$ in $\mathbb{R}^{n\geq 3}$ with the standard Euclidean metric
     $\mathsf{g}_0$. For the sake of convenience, let $t$ be the distance function of $o.$ To achieve its $(2,p)-$potential, we consider a radial function $f(t(\cdot))$ satisfying
    $$\Delta f=f'(t)\Delta t+f''=f'(t)\bigg(\frac{n-1}{t}\bigg)+f''(t)=(t^{n-1}f'(t))'t^{1-n}.$$
    In order to make $$\Delta(|\Delta f|^{p-2} \Delta f)=0,$$ we can set
    \begin{equation}\label{31}
      (t^{n-1}f'(t))'t^{1-n}=g(t)=
    \begin{cases}
    -(t^{2-n}-\mu^{2-n})^{\frac{1}{p-1}}\ & t\in[r,\mu];
    \\(\mu^{2-n}-t^{2-n})^{\frac{1}{p-1}} & t\in[\mu,R],
    \end{cases}\text{for some}\ \ \mu\in(r,R).
    \end{equation}
    Then
    \begin{equation}\label{32}
    \begin{cases}t^{n-1}f'(t)=\int_r^t s^{n-1}g(s)\,ds;
    \\f(t)=-\int_t^R\tau^{1-n}\int_r^\tau s^{n-1}g(s)\,dsd\tau.
    \end{cases}
    \end{equation}
   \begin{lemma}\label{l31}
     For any
     $$
     \begin{cases}
     n\ge 3;\\
     1<p<\infty;\\
     0<r<R<\infty,
     \end{cases}
     $$
     there is a unique $\mu\in(r,R)$ such that the above-constructed function $f:[r,R]\rightarrow\mathbb{R}$ satisfies
     \begin{itemize}
       \item [\rm (i)] $u(\cdot)=\frac{f(t(\cdot))}{f(r)}$ is the $(2,p)-$potential of $\big(\bar{B}(r),B(R)\big)$ in $\mathbb{R}^{n}.$
       \item [\rm (ii)] $${\rm cap}_{2,p}\big(\bar{B}(r),B(R)\big)={(n-2)\sigma_{n-1}}{\big(f(r)\big)^{1-p}}.$$
       \item [\rm (iii)] If $R>0$ is fixed, then $\mu$ is a smooth function of $r.$ Furthermore,
       $$\lim\limits_{r\rightarrow 0}\left(\frac{r}{\mu}\right)=0.$$
       \item [\rm (iv)] If $R>0$ is fixed, then
       $$\forall t\in[r,R],\ \ |f'(t)|\leq \left(\frac{p}{n(p-1)}\right)r^{\frac{2p-n}{p(p-1)}}t^{1-n}(t^{n-\frac np}-r^{n-\frac np}).$$
       \item [\rm (v)]
        $$f(r)=
       \begin{cases}\frac{p-1}{(n-2p)(n-2)}\left((r^{2-n}-\mu^{2-n})^{\frac{p}{p-1}}r^n-\Big(\mu^{2-n}-R^{2-n}\Big)^{\frac{p}{p-1}}R^n\right)&\text{for}\ \ p\neq\frac n2;
       \\ \frac{\ln r}{2-n}+{\it O}(1) &\text{for}\ \ p=\frac n2.
       \end{cases}$$
       \item [\rm (vi)] For any $R>0,$
       $$
       \begin{cases}
       f(r)\sim \frac{p-1}{(n-2p)(n-2)}r^{2-\frac{n-2}{p-1}}\rightarrow\infty  &\text{for}\ \  p\in(1,\frac{n}{2});
       \\ f(r)\sim \frac{\ln r}{2-n}  &\text{for}\ \ p=\frac n2;
       \\ f(r)\rightarrow \frac{p-1}{(2p-n)(n-2)}\Big (\frac{1}{c}-1\Big)^{\frac{p}{p-1}}R^{\frac{2p-n}{p-1}} &\text{for}\ \ p>\frac n2,
       \end{cases}\ \ \text{as}\ r\rightarrow 0.$$
       Here $c\in(0,1)$ is given by
       $$      \int_0^1 y^{\frac{2}{n-2}-\frac{1}{p-1}}(1-y)^{\frac{1}{p-1}}\, dy=
   \int^1_c y^{\frac{2n-2}{2-n}}(1-y)^{\frac{1}{p-1}}\, dy.$$
     \end{itemize}
   \end{lemma}

   \begin{proof} (i) By the construction of $f,$ we read off $$f(R)=f'(r)=0.$$ In order to show that $u(\cdot)$ is the $(2,p)-$potential of $\big(\bar{B}(r),B(R)\big),$
   we only need to check
   $$f'(R)=0.$$
   Hence we need to prove that
   $$
   \text{there exists a unique $\mu\in(r,R),$ such that $\int_r^R s^{n-1}g(s)\,ds=0.$}
   $$
   More concretely, we will show
   \begin{equation}\label{33}
 I_1=I_2\ \ \text{where}\ \  \begin{cases}
   I_1=\int_r^\mu s^{n-1}(s^{2-n}-\mu^{2-n})^{\frac{1}{p-1}}\,ds;\\  I_2=\int_\mu^R s^{n-1}(\mu^{2-n}-s^{2-n})^{\frac{1}{p-1}}\,ds.
   \end{cases}
   \end{equation}
   It is easy to see that
   $$\text{$I_1\ \&\ -I_2$ are smoothly increasing functions of $\mu$ with $I_1(r)=0=I_2(R)$.}
   $$
   Then we get the existence - uniqueness - smoothness of $\mu.$

   (ii) We utilize not only \eqref{110} for
   $$
   p\in (1,\infty)\ \ \&\ \
   (\mathbb M^n,\mathsf{g})=(\mathbb R^n,\mathsf{g}_0)
   $$
   but also the integration-by-parts to calculate
   $$
   \begin{aligned}
   {\rm cap}_{2,p}\big(\bar{B}(r),B(R)\big)&=\int_{B(R)\setminus B(r)}|\Delta u|^{p-2}\Delta u\cdot\Delta u\,d\upsilon_{{\mathsf{g}_0}}
   \\&=-\int_{\partial B(r)}\left(\frac{\partial(|\Delta u|^{p-2}\Delta u)}{\partial\nu}\right)\,d\sigma_{\mathsf{g}_0}
   \\&=-\big(f(r)\big)^{1-p}\int_{\partial B(r)}\frac{d}{dt}|g(t)|^{p-1}\,d\sigma_{\mathsf{g}_0}
   \\&={(n-2)\sigma_{n-1}}{\big(f(r)\big)^{1-p}},
   \end{aligned}$$
where
$$
\Big\{d\upsilon_{{\mathsf{g}_0}},  d\sigma_{{\mathsf{g}_0}}, \nu\Big\}
$$
is the triple of volume element, area element, unit outer normal in $(\mathbb R^n, \mathsf{g}_0)$.

   (iii) In order to control $\mu$ when $R$ is fixed, we will make an analysis for \eqref{33}. Upon setting
   $$s=\mu y^{\frac{1}{n-2}}\ \ \text{in}\ \ I_1\ \ \&\ \  s=\mu y^{\frac{1}{2-n}}\ \ \text{in}\ \ I_2\ \ \text{ seperately},
   $$
   we obtain
   $$J_1=\int_{\big(\frac{r}{\mu}\big)^{n-2}}^1 y^{\frac{2}{n-2}-\frac{1}{p-1}}(1-y)^{\frac{1}{p-1}}\, dy=
   \int^1_{\big(\frac{R}{\mu}\big)^{2-n}}y^{\frac{2n-2}{2-n}}(1-y)^{\frac{1}{p-1}}\, dy=J_2.$$

   Firstly, notice that
   $$\left(\frac{r}{\mu}\right)^{n-2}\left(\frac{R}{\mu}\right)^{2-n}=\left(\frac{r}{R}\right)^{n-2}\rightarrow 0\ \ \text{as}\ \ r\rightarrow 0.$$
   Then at least one of the lower limits of integrations tends to $0$ as $r\rightarrow 0.$ As a consequence, at least one of $J_1$ and $J_2$ approaches an improper integral as $r\rightarrow 0.$ Notice that the improper integral $J_2$ is always divergent as $$\frac{2n-2}{2-n}<-1,$$ while the convergence of the improper integral $J_1$ depends on the value of $$\frac{2}{n-2}-\frac{1}{p-1}.$$
   \par If $$\limsup\limits_{r\rightarrow 0}\left(\frac{r}{\mu}\right)= c>0 -\text{i.e.}- \exists\  r_k\rightarrow 0\ \ \text{as}\ \  k\rightarrow\infty\ \ \text{obeying}\ \
   \lim\limits_{k\rightarrow \infty}\left(\frac{r_k}{\mu_k}\right)= c>0,$$
   then
   $$\lim\limits_{k\rightarrow \infty}\left(\frac{R}{\mu_k}\right)^{2-n}=0,
   $$
   and hence
   $$
   \lim\limits_{k\rightarrow \infty} J_{2,k}=\infty\ \ \text{contradicting the finiteness of $J_1.$}
   $$
   Consequently, we obtain
   $$\lim\limits_{r\rightarrow 0}\left(\frac{r}{\mu}\right)=0.$$

   \par Furthermore, if
   $$p>2-\frac 2n,$$ then $$\frac{2}{n-2}-\frac{1}{p-1}>-1,
   $$ and hence the improper integral $J_1$ is convergent. Accordingly, there holds
    \begin{equation}\label{34}
    \begin{cases}\lim\limits_{r\rightarrow 0} J_2=\lim\limits_{r\rightarrow 0} J_1=\text{a constant};\\
    \lim\limits_{r\rightarrow 0}\left(\frac{R}{\mu}\right)^{2-n}=c;\\
      \int_0^1 y^{\frac{2}{n-2}-\frac{1}{p-1}}(1-y)^{\frac{1}{p-1}}\, dy=
   \int^1_c y^{\frac{2n-2}{2-n}}(1-y)^{\frac{1}{p-1}}\, dy.
   \end{cases}
    \end{equation}

   (iv) According to \eqref{32}, we have that if $t\in[r,R]$ then
   $$\begin{aligned}
   |t^{n-1}f'(t)|& \leq\int_r^t r^{\frac{n}{p}}s^{n-\frac np-1>-1}|g(s)|\,ds
   \\ & \leq r^{\frac{n}{p}}\max\limits_{s\in[r,R]}|g(s)|\int_r^ts^{n-\frac np-1}\,ds
   \\ & \leq r^{\frac{n}{p}}r^{\frac{2-n}{p-1}}\left(\frac{p}{n(p-1)}\right)s^{n-\frac np}|_r^t
   \\ & =\left(\frac{p}{n(p-1)}\right)r^{\frac{2p-n}{p(p-1)}}(t^{n-\frac np}-r^{n-\frac np})
   \end{aligned}
   $$

    (v) According to \eqref{32}, we compute
    $$\begin{aligned}
   (n-2)f(r)&=-(n-2)\int_r^R t^{1-n}\int_r^t s^{n-1}g(s)\,dsdt
   \\ &=t^{2-n}\int_r^t s^{n-1}g(s)\,ds\Bigl|_r^R-\int_r^Rt^{2-n}t^{n-1}g(t)\,dt
   \\ &=0-\int_r^Rt g(t)\,dt
   \\ &=I_3-I_4,
   \end{aligned}$$
   where
   $$
   \begin{cases} I_3=\int_r^\mu t(t^{2-n}-\mu^{2-n})^{\frac{1}{p-1}}\,dt;\\
   I_4=\int_\mu^R t(\mu^{2-n}-t^{2-n})^{\frac{1}{p-1}}\,dt.
   \end{cases}
   $$
   Upon recalling the above-verified formula $I_1=I_2$ and integrating by parts, we get
   $$\begin{aligned}
   I_3 & =\int_r^\mu \left(\frac{p-1}{(2-n)p}\right) t^n\Big((t^{2-n}-\mu^{2-n})^{\frac{p}{p-1}}\Big)'\, dt
   \\ & =\left(\frac{p-1}{(2-n)p}\right) t^n(t^{2-n}-\mu^{2-n})^{\frac{p}{p-1}}\Big|_r^\mu-\int_r^\mu \left(\frac{p-1}{(2-n)p}\right)(t^{2-n}-\mu^{2-n})^{\frac{p}{p-1}} nt^{n-1}\, dt
   \\ &=\left(\frac{p-1}{(n-2)p}\right) r^n(r^{2-n}-\mu^{2-n})^{\frac{p}{p-1}}+\left(\frac{n(p-1)}{(n-2)p}\right)(I_3-\mu^{2-n}I_1).
   \end{aligned}$$
   Similarly, we obtain
   $$\begin{aligned}
   I_4 & =\int_\mu^R \left(\frac{p-1}{(n-2)p}\right) t^n\Big((\mu^{2-n}-t^{2-n})^{\frac{p}{p-1}}\Big)'\, dt
   \\ & =\left(\frac{p-1}{(n-2)p}\right) t^n(\mu^{2-n}-t^{2-n})^{\frac{p}{p-1}}\Big|_\mu^R-\int_\mu^R \left(\frac{p-1}{(n-2)p}\right)(\mu^{2-n}-t^{2-n})^{\frac{p}{p-1}} nt^{n-1}\, dt
   \\ &=\left(\frac{p-1}{(n-2)p}\right) R^n(\mu^{2-n}-R^{2-n})^{\frac{p}{p-1}}+\left(\frac{n(p-1)}{(n-2)p}\right)(I_4-\mu^{2-n}I_2).
   \end{aligned}$$
   Hence we could get such a formula of $f(r)$ that if $p \neq \frac{n}{2}$ then
   \begin{equation*}
   \begin{aligned}
   f(r)&=\frac{I_3-I_4}{n-2}\\
   &=\left(\frac{p-1}{(n-2p)(n-2)}\right)\left((r^{2-n}-\mu^{2-n})^{\frac{p}{p-1}}r^n-\Big(\mu^{2-n}-R^{2-n}\Big)^{\frac{p}{p-1}}R^n\right).
   \end{aligned}
   \end{equation*}
    But nevertheless, if $p=\frac n2$ then we can achieve $f(r)$ through a direct calculation:
   $$\begin{aligned}
   f(r)=\frac{I_3-I_4}{n-2}=\frac{\ln r}{2-n}+{\it O}(1)\ \ \text{as}\ \ r\rightarrow 0.
   \end{aligned}$$
   Then (iv) is proved.

   (iv) It is a straightforward by-product of (iii)-(v)-\eqref{34}.

   \end{proof}

    Based on Lemma \ref{l31}, we can prove Theorem \ref{t41} in the sequel.

    \begin{proof}[Proof of Theorem \ref{t41}]
    Given $x\in \Omega$, we select a small $R>0$ such that $$B(x,R)\subseteq U\subseteq\Omega$$ where $U$ is the normal neighborhood of $x.$  In what follows, for a notational convenience, let $$t(\cdot)={\rm dist}_{\mathsf{g}}(x,\cdot)$$ be the distance function of $x$. Then $t(\cdot)$ is smooth in $B(x,R)\setminus\{x\}.$
     \par In the sequel, we use $C_1,C_2,\cdots$ to denote positive constants depending only on $\{n,p,R\}$.
    \par For $r\in (0,R)$ \& $f$ as defined in Lemma \ref{l31}, we choose the smooth function $$u(\cdot)=\frac{f(t(\cdot))}{f(r)}: B(x,R)\setminus B(x,r)\rightarrow \mathbb{R}$$
    to evaluate
    $$
    \begin{cases} u\big|_{\partial B(x,r)}=1;\\
     u\big|_{\partial B(x,R)}=0;\\
     |\nabla u|\big|_{\partial B(x,r)}=|\nabla u|\big|_{\partial B(x,R)}=0.
     \end{cases}
     $$
    Upon not only noticing
    $$
    \big|\Delta t-({n-1}){t}^{-1}\big|\leq C_1t\ \ \&\ \  |\partial B(x,t)|\leq C_1t^{n-1}\ \ \text{in}\ \ B(x,R)\setminus\{x\},
    $$
    but also using \eqref{31}'s function $g(t)$, we estimate
    $$\begin{aligned}
    {\rm cap}_{2,p}\big(\bar{B}(x,r),B(x,R)\big)&\leq \int_{ B(x,R)\setminus B(x,r)}|\Delta u|^p\,d\upsilon_{\mathsf{g}}
    \\ &=\frac{1}{\big(f(r)\big)^p}\int_r^R \int_{\partial B(x,t)}|f'(t)\Delta t+f''(t)|^p\,d\sigma_{\mathsf{g}} dt
    \\ &\leq \frac{C_2}{\big(f(r)\big)^p}\int_r^R \Big(|g(t)|^p+|tf'(t)|^p\Big)t^{n-1}\, dt,
    \end{aligned}
    $$
whence considering two aspects.
\begin{itemize}
    \item On the one hand, Lemma \ref{l31} (ii) derives
    $$ {\big(f(r)\big)^{-p}}\int_r^R |g(t)|^p \sigma_{n-1} t^{n-1}\, dt={\rm cap}_{2,p}\big(\bar{B}(r),B(R)\big)={(n-2)\sigma_{n-1}}{\big(f(r)\big)^{1-p}},
    $$
    where $\big(\bar{B}(r),B(R)\big)$ is the concentric round ring in $\mathbb{R}^n.$
    \item On the other hand, Lemma \ref{l31} (iv) ensures
    \begin{equation*}
    \begin{aligned}
    \int_r^R |tf'(t)|^pt^{n-1}\, dt&\leq \Big(\frac{p}{n(p-1)}\Big)^p \int_r^R r^{\frac{2p-n}{p-1}}t^{p(1-n)}(t^{n-\frac np}-r^{n-\frac np})^pt^pt^{n-1}\,dt
    \\ &\leq \Big(\frac{p}{n(p-1)}\Big)^p r^{2-\frac{n-2}{p-1}}\int_r^Rt^{2p-pn+n-1}(t^{n-\frac np})^p\,dt
    \\ &\leq C_3 r^{2-\frac{n-2}{p-1}},
    \end{aligned}
    \end{equation*}
    thereby leading to a consideration of next two situations.

    \begin{itemize}
      \item If $1<p<\frac{n}{2},$ then
      $$
      \begin{cases}
      f(r)\sim r^{2-\frac{n-2}{p-1}}\ \ \text{as}\ \ r\rightarrow 0;\\
    {\rm cap}_{2,p}\big(\bar{B}(x,r),B(x,R)\big)\leq C_2\Bigg(\frac{(n-2)\sigma_{n-1}}{\big(f(r)\big)^{p-1}}+\frac{C_3 r^{2-\frac{n-2}{p-1}}}{\big(f(r)\big)^p}\Bigg)=
    C_4r^{n-2p}.
    \end{cases}
    $$
      \item If $p>\frac{n}{2},$ then
      $$
      \begin{cases}
      f(r)\ \ \text{has a positive bound from below};\\
      r^{2-\frac{n-2}{p-1}}<1\ \ \text{for a sufficiently small}\ \ r;\\
    {\rm cap}_{2,p}\big(\bar{B}(x,r),B(x,R)\big)\leq C_2\Bigg(\frac{(n-2)\sigma_{n-1}}{\big(f(r)\big)^{p-1}}+\frac{C_3 r^{2-\frac{n-2}{p-1}}}{\big(f(r)\big)^p}\Bigg)\leq
    C_5.
    \end{cases}
    $$
    \end{itemize}
    As a conclusion, if $$p\neq\frac n2\ \ \&\ \  0< \alpha<1-\frac{2p}{n},$$ then
    $$0\leq\inf\limits_{K\Subset\Omega}\frac{{\rm cap}_{2,p}(K,\Omega)}{|K|^\alpha}\leq \inf\limits_{r>0}\frac{{\rm cap}_{2,p}(\bar{B}(x,r),B(x,R)}{|\bar{B}(x,r)|^\alpha}=0\ \ \text{as}\ \ |\bar{B}(x,r)|\sim r^n.
    $$
\end{itemize}
    Next, let us deal with the endpoint case $p=\frac n2$. For this we need to make a more accurate estimation.

    \begin{itemize}
    \item Firstly, we have
    $$\begin{aligned}
    |\Delta f(t)|^{\frac n2}&=|g(t)+f'(t)(\Delta t-(n-1)t^{-1})|^{\frac n2}\\
    &\leq \Big(|g(t)|+|C_1tf'(t)|\Big)^{\frac n2}\\
     &\leq|g(t)|^{\frac n2}+\sum\limits_{k=1}^n\binom{n}{k}|g(t)|^{\frac k2}|C_1tf'(t)|^{\frac {n-k}{2}}.
    \end{aligned}
    $$
    Notice that  $$|\partial B(x,t)|\leq \sigma_{n-1}t^{n-1}+C_6t^n\leq C_1t^{n-1}\ \ \text{in}\ \
    B(x,R)\setminus\{x\}.
    $$
    Then
        $$\begin{aligned}
    &{\rm cap}_{2,\frac n2}\big(\bar{B}(x,r),B(x,R)\big)\\
    &\ \ \leq \int_{ B(x,R)\setminus B(x,r)}|\Delta u|^{\frac n2}\,d\upsilon_{\mathsf{g}}\\
    &\ \ \leq \frac{1}{f(r)^{\frac n2}}\int_r^R \int_{\partial B(x,t)}\left(|g(t)|^{\frac n2}+C_7\sum\limits_{k=0}^{n-1}|g(t)|^{\frac k2}|tf'(t)|^{\frac {n-k}{2}}\right)\,d\sigma_{\mathsf{g}} dt\\
    &\ \ \leq \frac{\sigma_{n-1}}{f(r)^{\frac n2}}\int_r^R |g(t)|^{\frac n2}\frac{dt}{t^{1-n}}+\frac{C_6}{f(r)^{\frac n2}}\int_r^R |g(t)|^{\frac n2}\frac{dt}{t^{-n}}+\frac{C_1C_7}{f(r)^{\frac n2}}
    \sum\limits_{k=0}^{n-1}\int_r^R \left(\frac{|g(t)|^{\frac k2}}{|tf'(t)|^{\frac{k-n}{2}}}\right)\frac{dt}{t^{1-n}}\\
    &\ \ = M(r)+N(r)+L(r)
    \end{aligned}$$
    \item Secondly, from Lemma \ref{l31} it follows that
    $$M(r)={\rm cap}_{2,\frac n2}\big(\bar{B}(r),B(R)\big)=\frac{(n-2)\sigma_{n-1}}{\big(f(r)\big)^{\frac n2-1}}.
    $$
    Now, the L'hospital rule indicates
    $$N(r)=o(M(r)).
    $$
    Yet, for each
    $$k\in\big\{1,2,\cdots,n-1\big\}$$
    there is
    $$\begin{aligned}
    &\int_r^R |g(t)|^{\frac k2}|tf'(t)|^{\frac{n-k}{2}}t^{n-1}\,dt\\
    &\ \ \leq\left( \int_r^R \Big(|g(t)|^{\frac k2}t^{\frac{(n-1)k}{n}}\Big)^{\frac nk}\,dt\right)^{\frac kn}\left(\int_r^R\Big(|tf'(t)|^{\frac{n-k}{2}}t^{\frac{(n-1)(n-k)}{n}}\Big)^{\frac{n}{n-k}}\, dt\right)^{\frac{n-k}{n}}
    \\
    &\ \ =\left(\int_r^R |g(t)|^{\frac n2}t^{n-1} \right)^{\frac kn}\left(\int_r^R|tf'(t)|^{\frac{n}{2}}t^{n-1}\, dt\right)^{\frac{n-k}{n}}
    \\
    &\ \ \leq  \left(M(r)f(r)^{\frac n2}(\sigma_{n-1})^{-1} \right)^{\frac kn}\cdot (C_3)^{\frac{n-k}{n}}
    \end{aligned}$$
    Whenever $k=0,$ the previous estimation holds automatically due to \eqref{31}. Accordingly, we always have
    $$
    L(r)\leq C_8 \big(f(r)\big)^{\frac {k}n-\frac n2}.
    $$

   \item Thirdly, upon recalling
    $$
    f(r)=\frac{\ln r}{2-n}+{\it O}(1)\ \ \text{as}\ \ r\rightarrow 0,
    $$
    we obtain
    $${\rm cap}_{2,\frac n2}\big(\bar{B}(x,r),B(x,R)\big)\leq \sigma_{n-1}(n-2)^{\frac n2}(-\ln r)^{1-\frac n2}\Big(1+{\it o}(1)\Big)\ \ \text{as}\ \ r\rightarrow 0,
    $$
    which is equivalent to
    $$
    \exp\Big(-(\sigma_{n-1})^\frac{2}{n-2}(n-2)^\frac{n}{n-2}\big({\rm cap}_{2,\frac n2}\big(\bar{B}(x,r),B(x,R)\big)\big)^{\frac{2}{2-n}}\Big)\leq r^{1+o(1)}\ \text{as} \ r\rightarrow 0.
    $$
    \end{itemize}
    Consequently, for any $$0<\beta<n(\sigma_{n-1})^\frac{2}{n-2}(n-2)^\frac{n}{n-2},$$
    we arrive at the desired estimation
    $$\begin{aligned}
    0&\leq \inf\limits_{K\Subset \Omega}|K|^{-1}{{\exp\left(-\beta \big({\rm cap}_{2,\frac n2}(K,\Omega)\big)^{\frac{2}{2-n}}\right)}}\\
    &\leq
    \inf\limits_{r>0}|\bar{B}(x,r)|^{-1}{\exp\left(-\beta \big({\rm cap}_{2,\frac n2}(\big(\bar{B}(x,r),B(x,R)\big))\big)^{\frac{2}{2-n}}\right)}
    \\ &\leq \inf\limits_{r>0}|\bar{B}(x,r)|^{-1}{\Big(r^{1+o(1)}\Big)^{{\beta}{(\sigma_{n-1})^\frac{2}{2-n}(n-2)^\frac{n}{2-n}}}}
    \\ &=0.
    \end{aligned}
    $$
    \end{proof}

\section{Relationships with variational capacities}\label{s4}

This section explores how the second-order capacity ${\rm cap}_{2,1<p<\infty}(\cdot,\cdot)$ is related to the first-order capacity ${\rm cap}_{1<p<\infty}(\cdot,\cdot)$, thereby discovering the folllowing nice formula
$$
\underset{p\to 1}{\lim}{\rm cap}_{2,p}(\cdot,\cdot)={\rm cap}_{2,1}(\cdot,\cdot)=2{\rm cap}_{2}(\cdot,\cdot),
$$
for which the first equality is somewhat driven by \cite{Mey, Mo} while the second equality exists as a resolution to the open problem in Remark 2.16 of \cite{GPS} (whose main ideas have been extended by \cite{HS} to some discrete graphs).

    \begin{lemma}\label{l41} For $p\in [1,\infty)$ let $(K,\Omega)$ be a smooth condenser in $(\mathbb M^n,\mathsf{g})$ with the area element $d\sigma_{\mathsf{g}}$ \& the outer normal derivative $\partial/\partial\nu$. If
    $$
    h\in L^{\frac {p}{p-1}}(\Omega\setminus K)\ \ \&\ \ \Delta h=0,
    $$
    \begin{equation}
    \label{40}
    \big({\rm cap}_{2,p}(K,\Omega)\big)^{\frac 1p}\geq \Bigg|\int_{\partial\Omega}\left(\frac{\partial h}{\partial\nu}\right)\, d\sigma_{\mathsf{g}}\Bigg |\begin{cases} {\Big(\int_{\Omega\setminus K}|h|^{\frac {p}{p-1}}\, d\upsilon_{\mathsf{g}}\Big)^{\frac{1-p}p}}&\text{for}\ \ p\in (1,\infty);\\
    \|h\|^{-1}_{L^\infty(\Omega\setminus K)}&\text{for}\ \ p=1.
    \end{cases}
    \end{equation}
 Moreover, the above equality holds iff
 $$
 \begin{cases} h=\lambda |\Delta v|^{p-2}\Delta v\ \ \text{for some constant $\lambda\neq 0$};\\
 \text{$v$ is the $(2,p)-$potential of
    $(K,\Omega)$ with $|\Delta v|>0$ a.e. on $\Omega\setminus K$}.
\end{cases}
$$
    \end{lemma}
    \begin{proof} It is enough to verify \eqref{40} for $p\in (1,\infty)$. In fact, for any $v\in\mathfrak{C}^{1,1}(\Omega\setminus K)$ we use the H\"older inequality to calculate
      $$\begin{aligned}
      \Bigg|\int_{\partial\Omega}\Bigg(\frac{\partial h}{\partial\nu}\Bigg) d\sigma_{\mathsf{g}}\Bigg| &=\Bigg|\int_{\partial(\Omega\backslash K)}(1-v)\Bigg(\frac{\partial h}{\partial\nu}\Bigg) d\sigma_{\mathsf{g}}\Bigg|
      \\ &=\Bigg|\int_{\Omega\backslash K}\big((1-v)\Delta h-h\Delta(1-v)\big)\,d\upsilon_{\mathsf{g}}+\int_{\partial(\Omega\backslash K)}h\Bigg(\frac{\partial (1-v)}{\partial\nu}\Bigg) d\sigma_{\mathsf{g}}\Bigg|
      \\ &=\Bigg|\int_{\Omega\backslash K} (\Delta v)h\, d\upsilon_{\mathsf{g}}\Bigg|
      \\ &\leq\int_{\Omega\backslash K} |h\Delta v|\,d\upsilon_{\mathsf{g}}
      \\ &\leq\left(\int_{\Omega\backslash K} |h|^{\frac{p}{p-1}}\,d\upsilon_{\mathsf{g}}\right)^{1-\frac 1p}\left(\int_{\Omega\backslash K} |\Delta v|^p\,d\upsilon_{\mathsf{g}}\right)^{\frac 1p}
      \end{aligned}
      $$
      Clearly, {if the equality in the last inequality is achieved, then
      $$
      h\Delta v\  \text{is either non-negative or non-positive with}\ |h|^{\frac{p}{p-1}}=\lambda |\Delta v|^p\ \text{in}\ \Omega\backslash K.
      $$}
      Taking the infimum over all $v\in\mathfrak{C}^{1,1}(\Omega\setminus K)$ completes the argument.
    \end{proof}

    \begin{theorem}\label{t42} Under the same hypothesis as in Lemma \ref{l41} there holds

    \begin{equation}
    \label{400}
    \big({\rm cap}_{2,p}(K,\Omega)\big)^{\frac 1p}\geq \frac{{\rm cap}_2(K,\Omega)}{\big(m_p(K,\Omega)\big)^{\frac{p-1}{p}}},
    \end{equation}
    where
    $$
    \begin{cases}
    \big(m_p(K,\Omega)\big)^\frac{p-1}{p}=\begin{cases} \min\limits_{c\in\mathbb{R}}\Big(\int_{\Omega\setminus K}|u-c|^{\frac{p}{p-1}}\, d\upsilon_{\mathsf{g}}\Big)^\frac{p-1}{p}&\text{for}\ \ p\in (1,\infty);\\
    {\min\limits_{c\in\mathbb{R}}\|u-c\|_{L^\infty(\Omega\setminus K)}}&\text{for}\ \ p=1,
    \end{cases}\\
    \text{$u$ is the $2-$capacity potential of $(K,\Omega).$}
    \end{cases}
    $$
    Moreover, the inequality in \eqref{400} becomes an equality iff
    $$\Delta v\big|_{\partial K}\ \ \&\ \ \Delta v\big|_{\partial \Omega}$$ are constants for the $(2,p)-$potential function $v$ of $(K,\Omega).$
    \end{theorem}

    \begin{proof} It suffices to demonstrate \eqref{400}'s case $p\in (1,\infty)$ in the sequel.
    	
     Firstly, recall that if $u$ is the $2-$capacity potential of $(K,\Omega)$ - that is -
     $$\begin{cases}
     \Delta u=0 \ & \text{in}\ \Omega\setminus K; \\
     u=1 & \text{on } \ \ \partial K;\\
     u=0 & \text{on } \ \ \partial \Omega,
     \end{cases}
     $$
     then
     $$\text{not only $u$ is smooth but also $
     {\rm cap}_2(K,\Omega)=\int_{\Omega\setminus K}|\nabla u|^2\, d\upsilon_{\mathsf{g}}=\int_{\partial\Omega}\Bigg(\frac{\partial u}{\partial\nu}\Bigg)\, d\sigma_{\mathsf{g}}$},
     $$
and  hence
     $$\text{Lemma \ref{l41}}\Longrightarrow\big({\rm cap}_{2,p}(K,\Omega)\big)^{\frac 1p}\geq \frac{{\rm cap}_2(K,\Omega)}{\Big(\int_{\Omega\setminus K}|u-c|^{\frac {p}{p-1}}\, d\upsilon_{\mathsf{g}}\Big)^{1-\frac 1p}}\ \ \forall\ \ c\in\mathbb R.
     $$

    \par Secondly, $m_p(K,\Omega)$ is well-defined as the minimum can be obtained by noticing that the function $$\mathbb R\ni c\mapsto \int_{\Omega\setminus K}|u-c|^{\frac{p}{p-1}}\, d\upsilon_{\mathsf{g}}$$
     not only is continuous but also tends to $\infty$ as $c\rightarrow\infty.$ Meanwhile, it is easy to find that the minimum point of this function is in the interval $(0,1).$
     \par Thirdly, if the equality of the last inequality is achieved, then
     $$
     u-c=\lambda|\Delta v|^{p-2}\Delta v\ \ \text{
     for some constants $c$ \& $\lambda\neq 0.$},
 $$
 and hence
 $$\Delta v\big|_{\partial K}\ \ \&\ \ \Delta v\big|_{\partial \Omega}\ \ \text{are constants.}
 $$
 Meanwhile, if $$\Delta v\big|_{\partial K}\equiv a\neq b\equiv\Delta v\big|_{\partial \Omega},$$ then setting
     $$\lambda=(|b|^{p-2}b-|a|^{p-2}a)^{-1} \ \&\ c=-|a|^{p-2}a\lambda$$
      would lead to the equality.
 \end{proof}
    \begin{remark}\label{r43} If
      $$
      \begin{cases}
      1<p<\infty;\\
      1<q=\frac{p}{p-1}<\infty;\\
      f(c)=\int_{\Omega\setminus K}|u-c|^{q}\, d\upsilon_{\mathsf{g}},
      \end{cases}
      $$
      then this last function $f$ is not only differential but also convex, and hence its minimum point $c_0$ is unique - in particular - if $p=q=2$ then
            $$c_0=\frac{\int_{\Omega\setminus K} u\, d\upsilon_{\mathsf{g}}}{|\Omega|-|K|}.$$
  \end{remark}

   Notably, we can answer \cite[Remark 2.16]{GPS}'s interesting question via the following assertion.

   \begin{theorem}\label{t22}
   	If $(K,\Omega)$ is a smooth condenser in $(\mathbb M^n,\mathsf{g})$, then \eqref{111} holds.
   	\end{theorem}
   \begin{proof} Two circumstances must be handled separately.
   	\begin{itemize}
   		\item If $p=1$, then Theorem \ref{t42} yields
   	\begin{equation}
   	\label{41}
   	{\rm cap}_{2,1}(K,\Omega)\geq \frac{{\rm cap}_2(K,\Omega)}{\min\limits_{c\in\mathbb{R}}\|u-c\|_{L^\infty(\Omega\setminus K)}}=2{\rm cap}_2(K,\Omega).
   	\end{equation}
   	Thus, in order to reach \eqref{111}'s case $p=1$, we only need to prove \eqref{41}'s reversed inequality
   	\begin{equation}
   	\label{42}
   	{\rm cap}_{2,1}(K,\Omega)\leq 2{\rm cap}_2(K,\Omega).
   	\end{equation}
   	
   	In doing so, we may assume that
   	$u$ is the $2-$potential of $(K,\Omega)$, whence trying to show that $u$ is ``nearly" the $(2,1)-$potential of $(K,\Omega).$ To be more precise, note that
   	$$
   	\text{$u$ is smooth in $\overline{\Omega\setminus K}$ with
$\begin{cases}
   	u|_{\partial K}=1;\\
   	u|_{\partial\Omega}=0;\\
   	|\nabla u|\big|_{\partial(\Omega\setminus K)}\neq 0.
   	\end{cases}$}
   	$$
   	So, we are going to make a transformation of $u.$ 	
   	
   	For some small $\delta\in (0,{2^{-1}}]$, define not only a ``bump" function:
   	\begin{equation}\label{43}
   	\begin{cases}
   	f\in C^2[0,{2^{-1}}];\\
   	f(0)=f'(0)=f'(\delta)=f''(\delta)=0;\\
   	f\big|_{[\delta,{2^{-1}}]}=1,
   	\end{cases}
   	\end{equation}
   	but also
   	$$
   	\begin{cases} f(t)=f(1-t)$ for $t\in ({2^{-1}},1];\\
   	w=\begin{cases}
   	uf(u) & \text{if} \ \ u\in[0,{2^{-1}}]; \\
   	uf(u)+1-f(u)& \text{if}\ \ u\in ({2^{-1}},1].
   	\end{cases}
   	\end{cases}
   	$$
   	Then $$w\in \mathfrak{C}^{2}(\Omega\setminus K){\subseteq\mathfrak{C}^{1,1}(\Omega\setminus K)}.$$
   	A direct calculation indicates that not only
   	$$
   	\Delta w=\begin{cases}
   	\big(uf(u)\big)''|\nabla u |^2& \text{if}\ \ u\in[0,{2^{-1}}]; \\
   	\big((u-1)f(u)\big)''|\nabla u |^2 & \text{if}\ \ u\in ({2^{-1}},1],
   	\end{cases}
   	$$
   	but also
   	$$
   	\begin{aligned}
   	&\int_{\Omega\setminus K}|\Delta w|\, d\upsilon_{\mathsf{g}}\\
   	&\ \ = \int_0^1 \int_{\{u=t\}}\frac{|\Delta w|}{|\nabla u|}\,d\sigma_{\mathsf{g}}\,dt
   	\\
   	&\ \ =\int_0^{{2^{-1}}} |\big(tf(t)\big)''|\int_{\{u=t\}}|\nabla u|\,d\sigma_{\mathsf{g}}\,dt+\int_{{2^{-1}}}^1|\big((t-1)f(t)\big)''|\int_{\{u=t\}}|\nabla u|\,
   	d\sigma_{\mathsf{g}}\,dt\\
   	&\ \ ={\rm cap}_2(K,\Omega)\left(\int_0^{{2^{-1}}}|\big(tf(t)\big)''|dt+\int_{{2^{-1}}}^1 |\big((t-1)f(t)\big)''|\,dt\right)\\
   	&\ \ =2{\rm cap}_2(K,\Omega) \int_0^{{2^{-1}}} |\big(tf(t)\big)''|dt
   	\end{aligned}
   	$$
   	{Despite having
    $$\int_0^{{2^{-1}}} |\big(tf(t)\big)''|\, dt\geq \big(tf(t)\big)'\big|_0^{2^{-1}}=1,$$
    we are able to }construct a special function $f$ such that the last integral is nearly $1.$ If we set $$g(t)=tf(t),$$ then \eqref{43} is equivalent to
   	$$
   	\begin{cases} g\in C^2[0,{2^{-1}}];\\
   	g(0)=g'(0)=g''(0)=0;\\
   	g(\delta)=\delta;\\
   	g'(\delta)=1;\\
   	g''|_{[\delta,{2^{-1}}]}=0.
   	\end{cases}    	
   	$$
   	Given $\{M,m,j,k\},$ we select the continuous function $g''$ as
   	$$
   	g''(t)=\begin{cases}
   	M\sin(kt)& \text{if}\ \ t\in[0,\frac{\pi}{k}]; \\
   	-m\sin (jt-\frac{j\pi}{k})& \text{if}\ \ t\in[\frac{\pi}{k},\frac{\pi}{k}+\frac{\pi}{j}]; \\
   	0 &\text{if}\ \ t\in[\frac{\pi}{k}+\frac{\pi}{j},{2^{-1}}],
   	\end{cases}
   	$$
   	whence getting
   	$$ g'(t)=\int_0^t g''(s)ds=\begin{cases}
   	\frac{M}{k}\Big(1-\cos(kt)\Big)& \text{if}\ \ t\in[0,\frac{\pi}{k}]; \\
   	\frac{2M}{k}+\frac{m}{j}\Big(\cos (jt-\frac{j\pi}{k})-1\Big) & \text{if}\ \ t\in[\frac{\pi}{k},\frac{\pi}{k}+\frac{\pi}{j}]; \\
   	\frac{2M}{k}-\frac{2m}{j} & \text{if}\ \ t\in[\frac{\pi}{k}+\frac{\pi}{j},{2^{-1}}].
   	\end{cases}
   	$$
   	For any sufficiently small $\varepsilon>0,$ let
   	$$
   	\begin{cases}
   	j=\frac{1}{\varepsilon};\\
   	m=1;\\
   	k=\frac{1-2\varepsilon}{2\varepsilon^2};\\
   	M=({2^{-1}}+\varepsilon)k;\\ \delta=\frac{\pi}{k}+\frac{\pi}{j}\ll{2^{-1}}.
   	\end{cases}
   	$$
   	Then we deduce not only
   	$$
   	g'(\delta)=\frac{2M}{k}-\frac{2m}{j}=1,
   	$$
   	but also
   	$$\begin{aligned}
   	g(\delta)&=\int_0^{\frac{\pi}{k}} \frac{M}{k}\Big(1-\cos(kt)\Big) dt+\int_{\frac{\pi}{k}}^{\frac{\pi}{k}+\frac{\pi}{j}}\left(\frac{2M}{k}+\frac{m}{j}\Big(\cos (jt-\frac{j\pi}{k})-1\Big)\right)\,dt
   	\\ &=\frac{M\pi}{k^2}+\Big(\frac{2M}{k}-\frac{m}{j}\Big)\frac{\pi}{j}\\
   	&=\delta.
   	\end{aligned}
   	$$
   	In the end, we reach
   	$$\int_0^{{2^{-1}}} |g''(t)|dt= \frac{2M}{k}+\frac{2m}{j}=1+4\varepsilon,
   	$$
   	thereby obtaining
   	$$
   	{\rm cap}_{2,1}(K,\Omega)\leq\int_{\Omega\setminus K}|\Delta w|\, d\upsilon_{\mathsf{g}} \leq 2{\rm cap}_2(K,\Omega)(1+4\varepsilon),
   	$$
   	which, along with letting $\varepsilon\to 0$, gives the required inequality \eqref{42}.

   \item If $p\in (0,1)$, then letting not only $u$ be the $2-$capacity potential of $(K,\Omega)$
    	but also
    	$$\Big\{f,w,g, M, m, j,k, \varepsilon\Big\}$$ be defined as above, we employ the classic co-area formula based on the area element $d\sigma_{\mathsf{g}}$ of $(\mathbb M^n,\mathsf{g})$ to compute
    	$$\begin{aligned}
    	&\int_{\Omega\setminus K}|\Delta w|^p\, d\upsilon_{\mathsf{g}}\\
    	&\ \ = \int_0^1 \int_{\{u=t\}}\frac{|\Delta w|^p}{|\nabla u|}\,d\sigma_{\mathsf{g}}\,dt\\
    	&\ \ =\int_0^{{2^{-1}}} |\big(tf(t)\big)''|^p\int_{\{u=t\}}\frac{d\sigma_{\mathsf{g}}\,dt}{|\nabla u|^{1-2p}}+\int_{{2^{-1}}}^1|\big((t-1)f(t)\big)''|^p \int_{\{u=t\}}\frac{d\sigma_{\mathsf{g}}\,dt}{|\nabla u|^{1-2p}}\\
    	&\ \  \leq 2\left(\max\limits_{t\in [0,1]}\int_{\{u=t\}}|\nabla u|^{2p-1}\,d\sigma_{\mathsf{g}}\right)\int_0^{{2^{-1}}} |g''(t)|^p\, dt
    	\end{aligned}
    	$$
    	Notice that
    	$$u\ \ \text{is smooth and}\ \
    	\nabla u\neq 0\ \ \text{in}\ \ \overline{\Omega\setminus K}.
    	$$
    	Thus
    	$$\max\limits_{t\in [0,1]}\int_{\{u=t\}}|\nabla u|^{2p-1}\,d\sigma_{\mathsf{g}}<\infty
    	$$
    	while
    	$$\begin{aligned}
    	\int_0^{{2^{-1}}} |g''(t)|^p\, dt & =\int_0^{\frac{\pi}{k}}\Big(M\sin(kt)\Big)^p\,dt+\int_{\frac{\pi}{k}}^{\frac{\pi}{k}+\frac{\pi}{j}}\Big(m\sin(jt-\frac{j\pi}{k})\Big)^p\,dt
    	\\ & \leq \int_0^{\frac{\pi}{k}}\Big(Mkt)\Big)^p\,dt+ \int_0^{\frac{\pi}{j}}\Big(mjt)\Big)^p\,dt
    	\\ &=\frac{\pi^{p+1}}{p+1}\Bigg(\frac{M^p}{k}+\frac{m^p}{j}\Bigg)
    	\\ &=\frac{\pi^{p+1}}{p+1}\left(({2^{-1}}+\varepsilon)^p\bigg(\frac{1-2\varepsilon}{2\varepsilon^2}\bigg)^{p-1}+\varepsilon\right)\\
    	&\rightarrow 0\ \ \text{as}\ \  \varepsilon\rightarrow 0\ \ \text{since}\ \ p<1.
    	\end{aligned}
    	$$
    	Consequently, we get
    	$$
    	{\rm cap}_{2,p}(K,\Omega)=0.
    	$$
    \end{itemize}

   \end{proof}

     Theorem \ref{t22} \& its proof are fortunately extendable to the functional/variational capacity ${\rm cap}_p(\cdot,\cdot)$ on $(\mathbb M^n,\mathsf{g})$ whose Euclidean case can be found in either \cite{LiHo} for $p\in (1,\infty)$ or \cite{BP} for $p=2$.
    \begin{corollary}
    For any $p\in (1,\infty)$ and smooth condenser $(K,\Omega)$ in $(\mathbb M^n,\mathsf{g}),$ let
    $$
    \begin{cases} {\rm C}_{\Delta_p}(K,\Omega)=\inf\Big\{\int_{\Omega\setminus K}|\Delta_pf|\,d\upsilon_{\mathsf{g}}:\ f\in \mathfrak{C}^{1,1}(\Omega\setminus K)\Big\};\\
    \Delta_p f=\langle\nabla,|\nabla f|^{p-2}\nabla f\rangle=\text{div}\big(|\nabla f|^{p-2}\nabla f\big).
    \end{cases}
    $$
    Then
     \begin{equation}
     \label{eq42o}
       {\rm C}_{\Delta_p}(K,\Omega)=2{\rm cap}_p(K,\Omega).
      \end{equation}
    \end{corollary}
    \begin{proof} Below is a two-fold argument.
    	\begin{itemize}
        \item On the one hand, for any $f\in \mathfrak{C}^{1,1}(\Omega\setminus K),$ we estimate
        $$
        \begin{aligned}
        &{2^{-1}}\int_{\Omega\setminus K}|\Delta_pf|\,d\upsilon_{\mathsf{g}}\\
        &\ \ \geq \int_{\Omega\setminus K}({2^{-1}}-f)\Delta_pf\,d\upsilon_{\mathsf{g}}
        \\
        &\ \ =\int_{\partial(\Omega\setminus K)} ({2^{-1}}-f)|\nabla f|^{p-2}\left(\frac{\partial f}{\partial\nu}\right)\, d\sigma_{\mathsf{g}}-\int_{\Omega\setminus K}|\nabla f|^{p-2}\big\langle\nabla({2^{-1}}-f),\nabla f\big\rangle\, d\upsilon_{\mathsf{g}}
        \\
        &\ \ =\int_{\Omega\setminus K}|\nabla f|^{p}\, d\upsilon_{\mathsf{g}}\\
        &\ \ \geq {\rm cap}_p(K,\Omega),
        \end{aligned}
        $$
        whence
        \begin{equation}
        \label{44}
        {\rm C}_{\Delta_p}(K,\Omega)\geq 2{\rm cap}_p(K,\Omega).
        \end{equation}

        \item On the other hand, let not only $u$ be the $p-$potential of $(K,\Omega)$ but also
        $$\bigg\{w,\ f,\ g(t)=tf(t)\bigg\}
        $$
        be defined as in the proof of Theorem \ref{t22}. Then
          $$
        \Delta_p w= \Big([g'(u)]^{p-1}\Big)'|\nabla u |^p\ \ \text{
        when $u\in [0,{2^{-1}}]$ since $g'>0.$}
    $$
    Then
        $$
         \begin{aligned}
      \int_{\Omega\setminus K}|\Delta_p w|\, d\upsilon_{\mathsf{g}} & = 2{\rm cap}_p(K,\Omega) \int_0^{{2^{-1}}}\Bigg| \Big(\big(g'(t)\big)^{p-1}\Big)'\Bigg|dt
      \\ &=  2{\rm cap}_p(K,\Omega)\cdot\Bigg(\big(g'(t)\big)^{p-1}\Big|_{0}^{\frac{\pi}{k}}-\big(g'(t)\big)^{p-1}\Big|_{\frac{\pi}{k}}^1\Bigg)
      \\ &= 2{\rm cap}_p(K,\Omega)\cdot\Bigg(2\Big(g'\big(\frac{\pi}{k}\big)\Big)^{p-1}-1\Bigg)
      \\ &=2{\rm cap}_p(K,\Omega)\cdot\Big(2(1+2\varepsilon)^{p-1}-1\Big)
      \end{aligned}
        $$
        By sending $\varepsilon\rightarrow 0$ in the above formula, we get
       \begin{equation}
       \label{eq42b}
       {\rm C}_{\Delta_p}(K,\Omega)\leq 2{\rm cap}_p(K,\Omega).
       \end{equation}
\end{itemize}
       Now a combination of \eqref{44}-\eqref{eq42b} derives the required identification \eqref{eq42o}.
    \end{proof}

\section{Capacities under conformal changes}\label{s5}

Naturally, a combination of \eqref{16} -- \eqref{19} leads to
    \begin{definition}\label{d51}
      For $p\in [1,\infty)$ and any smooth compact domain $K$ {in a given non-compact Riemannian manifold} $(\mathbb M^n,\mathsf{g})$, let
     $$
     \begin{cases}{\rm cap}_{2,p}(K;\mathsf{g})=\inf\Big\{\int_{K^c}|\Delta f|^p\,d\upsilon_{\mathsf{g}}:\ f\in\mathfrak{C}^{1,1}(K^c)\Big\};\\
\mathfrak{C}^{1,1}(K^c)=\Big\{f\in C^{1,1}(\overline{K^c}):\ f\big|_{\partial K}-1=0=f\big|_{\infty}\ \&\ |\nabla f|\big|_{\partial K}=0=|\nabla f|\big|_{\infty}\Big\}.
    \end{cases}
    $$
 {A function $$u\in\mathfrak{C}^{1,1}(K^c)+W^{2,p}_0(K^c)$$ is called the $(2,p)-$potential of $K$ if $u$ is a weak solution to
    \eqref{18} (where $\Omega=\mathbb M^n$) in the sense that
    $$
    \int_{K^c}|\Delta u|^{p-2} \Delta u\Delta v\,d\upsilon_{\mathsf{g}}=0\ \ \forall\ \ v\in W^{2,p}_0(K^c).
    $$}
    \end{definition}

     Note that if $p\in (1,\infty)$ then the $(2,p)-$potential is always unique (whenever it exists) for any smooth compact domain $K$ as the $p-$energy $\mathsf{E}_p(\cdot)$ is strict convex. However, the existence of $(2,p)-$potential may depend on the geometric property of $(\mathbb M^n,\mathsf{g})$ such as the volume growth rate and the asymptotical behavior of its Green function of $p-$biharmonic operator. Therefore we're not here to conduct an in-depth research except finding an unexpected conformal structure of \eqref{12}-\eqref{14} as described below.

    \begin{theorem}\label{t51} For any smooth compact subdomain $K$ of $\mathbb R^n$, let
    	$$
    	{\rm Conf}(K)=\left\{g=\begin{cases}
    	e^{2u}\mathsf{g}_0& \text{for}\ \ n=2
    	\\ u^{\frac{4}{n-2}}{\mathsf{g}_0}  & \text{for}\ \ n\geq 3\ \&\ u>0
    	\end{cases}\right\}.
    	$$
    	\begin{itemize}
   	\item [\rm (i)] If
   	$$
   	c_n= \begin{cases}
   	{2^{-1}} \ & \text{for}\ n=2;\\
   	\Big(\frac{n-2}{4(n-1)}\Big)^{\frac{2n}{n+2}} & \text{for}\ n\geq 3,
   	\end{cases}
   	$$
   	then
    \begin{equation}
    \label{52}
    {\rm cap}_{2,\frac{2n}{n+2}}(K;{\mathsf{g}_0})=\inf\limits_{g\in{\rm Conf}(K)\ \&\ u\in \mathfrak{C}^{1,1}(K^c)}\left(c_n\int_{K^c} |{\mathsf{R_g}}|^{\frac{2n}{n+2}}\, d\upsilon_{\mathsf{g}}\right).
    \end{equation}

   \item [\rm (ii)] If not only $\mathsf{g}\in {\rm Conf}(K)$ but also $(\mathbb R^n,\mathsf{g})$'s scalar curvature $\mathsf{R_{g}}=0$, then
    \begin{equation}
    \label{53}
    {\rm cap}_{2,\frac{2n}{n+2}}(K;\mathsf{g})={\rm cap}_{2,\frac{2n}{n+2}}(K;{\mathsf{g}_0})\ \ \text{as $u\to 1$ at infinity}.
    \end{equation}
    \end{itemize}
    \end{theorem}
    \begin{proof} (i) Two situations are handled in the sequel.
\begin{itemize}    	
    \item  If $n\geq 3,$ then for
     $$
     \begin{cases} u\in\mathfrak{C}^{1,1}(K^c);\\
     \mathsf{g}=u^{\frac{4}{n-2}}{\mathsf{g}_0},
     \end{cases}
     $$
     we utilize the transformation law of conformal deformations to calculate
      $$
        {\mathsf{R_g}}=u^{-\frac{n+2}{n-2}}\Bigg(\mathsf{R}_{{\mathsf{g}_0}}u-\frac{4(n-1)}{n-2}\Delta_{{\mathsf{g}_0}} u\Bigg).
      $$
      Since $R_{{\mathsf{g}_0}}=0$, we derive
      $$
        \int_{K^c}|\Delta_{{\mathsf{g}_0}} u|^p\,d\upsilon_{\mathsf{g}_0}= \int_{K^c}\Bigg|\frac{n-2}{4(n-1)}{\mathsf{R_g}} u^{\frac{n+2}{n-2}}\Bigg|^p\cdot u^{\frac{-2n}{n-2}}\,d\upsilon_{\mathsf{g}},
      $$
thereby getting \eqref{52} under $n\ge 3$ via choosing $p= \frac{2n}{n+2}.$

      \item If $n=2,$ then letting $g= e^{2u}\mathsf{g}_0$ derives not only
      $${\mathsf{R_g}}=e^{-2u}({\mathsf R}_{{\mathsf{g}_0}}-2\Delta_{\mathsf{g}_0} u)
      $$
      but also
       $$
        \int_{K^c}|\Delta_{\mathsf{g}_0} u|\,d\upsilon_{{\mathsf{g}_0}}= \int_{K^c}\Big|{2^{-1}} e^{2u} {\mathsf{R_g}}\Big|e^{-2u}\,d\upsilon_{\mathsf{g}}=
        {2^{-1}} \int_{K^c}|{\mathsf{R_g}}|\,d\upsilon_{\mathsf{g}},
      $$
      whence establishing \eqref{52} under $n=2$.
      \end{itemize}
      (ii) Also, two situations are dealt with in the sequel.
      \begin{itemize}
      	\item
      If $n\ge 3$, then we check the desired formula \eqref{53} according to three steps.

      Firstly, the set $\mathfrak{C}^{1,1}(K^c)$ stays unchanged under the conformal deformations as
        $$
        |\nabla f|=0\Longleftrightarrow|\nabla_{{\mathsf{g}_0}} f|=0.
        $$

        Secondly, we know that
        $$\Delta_{{\mathsf{g}_0}}u=0\ \ \text{in}\ \ K^c\ \ \text{thanks to}\ \  {\mathsf{R_g}}=0,
        $$
        so that a gradient estimation gives
        $$
        |\nabla u|\big|_\infty=0.
        $$
        As an application, there holds
        $$\begin{aligned}
        v\in \mathfrak{C}^{1,1}(K^c)& \Longleftrightarrow \begin{cases} (1-v)\big|_{\partial K}=0;\\
        (1-v)\big|_{\infty}=1;\\
        |\nabla_g(1-v)|\big|_{\partial K}=0=|\nabla_g(1-v)|\big|_{\infty}.
        \end{cases}\\
        &
        \Longleftrightarrow\begin{cases} \frac{1-v}{u}\big|_{\partial K}=0;\\
        \frac{1-v}{u}\big|_{\infty}=1;\\
        \big|\nabla\big(\frac{1-v}{u}\big)\big|\bigg|_{\partial K}=0=\big|\nabla \big(\frac{1-v}{u}\big)\big|\bigg|_{\infty}.
        \end{cases}\\
        &\Longleftrightarrow 1-\frac{1-v}{u}\in \mathfrak{C}^{1,1}(K^c).
        \end{aligned}
        $$

        Thirdly, we utilize
        $$
             \Delta \Bigg(1-\frac{1-v}{u}\Bigg)=u^{-\frac{n+2}{n-2}}\Delta_{\mathsf{g}_0} (1-v)
        $$
        to calculate
        $$
        \int_{K^c} \left|\Delta \Bigg(1-\frac{1-v}{u}\Bigg)\right|^{\frac{2n}{n+2}}\,d\upsilon_{\mathsf{g}}= \int_{K^c} |\Delta_{{\mathsf{g}_0}}v|^{\frac{2n}{n+2}}\,d\upsilon_{{\mathsf{g}_0}},
        $$
        which, via taking the infimum over $v\in \mathfrak{C}^{1,1}(K^c)$, yields the desired formula \eqref{53} for $n\ge 3$.

        \item If $n=2$, then for any smooth function $v$ on $\Omega\setminus K\subseteq \mathbb R^2,$ we have
        $$
        \int_{\Omega\setminus K}|\Delta v|\, d\upsilon_{\mathsf{g}}=\int_{\Omega\setminus K}|\Delta_{\mathsf{g_0}}v|\, d\upsilon_{\mathsf{g}_0},
        $$
        thereby getting
        $${\rm cap}_{2,1}(K,\Omega;\mathsf{g})={\rm cap}_{2,1}(K,\Omega;\mathsf{g}_0).$$
        In other words, the identification
        $${\rm cap}_{2,1}(K,\Omega)=2{\rm cap}_{2}(K,\Omega)$$ is always conformally invariant.
        Needless to say, the desired formula \eqref{53} for $n=2$ is a special case of the just-verified invariance.

        \end{itemize}
        \end{proof}

    \section{ADM-type masses of asympotically flat manifolds}\label{s6}

    In accordance with Schoen's paper \cite{S}, given an integer $n\ge 3$, an $n$-dimensional complete Riemannian manifold
    $(\mathbb M^n,\mathsf{g})$ is called asymptotically flat (AF) provided that there are not only a compact subset $K$ of $\mathbb M^n$ but also a diffeomorphism $$
    \Phi:\ \mathbb M^n\setminus K\to\{x\in\mathbb R^n:\ |x|>1\}=\mathbb R^n\setminus \bar{B}(1)
    $$
    such that
    $$
    \begin{cases}
    {\mathsf g}_{ij}(x)=\delta_{ij}+{\it O}(|x|^{-\alpha});\\
    |x|{\mathsf g}_{ij,k}(x)+|x|^2|g_{ij,kl}(x)={\it O}(|x|^{-\alpha});\\
    |\mathsf{R_g}(x)|={\it O}(|x|^{-\beta});\\
    \text {for a parameter pair}\ \{\alpha,\beta\}\ \text{with}\ \min\big\{2(1+\alpha), \beta\big\}>n,
    \end{cases}
    $$
    where $\mathsf{R_g}$ still represents the scalar curvature of $(\mathbb M^n,\mathsf{g})$.

    Schoen-Yau's \cite[Theorem 5.3]{SY} (cf. \cite{SY1, SY0}) reveals that if $n\le 7$ or $(\mathbb{M}^{n},{\mathsf{g}})$ admits a Witten's spinor $\psi$ as described in \cite{Wi}, then the scalar curvature requirement
    $$0\le \mathsf{R}_g\in L^1(\mathbb M^n)$$ yields that the $\mathrm{ADM}$ mass (cf. Arnowitt-Deser-Misner's work \cite{ADM}) of $(\mathbb{M}^{n},{\mathsf{g}})$ is
    \begin{equation}
    \label{61}
    \mathsf{M}_{\mathrm{ADM}}(\mathbb{M}^{n},{\mathsf{g}})=\big({2(n-1)\sigma_{n-1}}\big)^{-1}\underset{r\to\infty}{\lim}\int_{\mathsf{S}_r}\sum_{j=1}^n\sum_{i=1}^n{({\mathsf{g}}_{ij,i}-{\mathsf{g}}_{ii,j})\nu_j}\,d\mathsf{S}\ge 0,
    \end{equation}
    where not only $\mathsf{S}_r$ is the coordinate sphere of radius $r$, but also $\nu=(\nu_1,...,\nu_n)$ is the outward unit normal to $\mathsf{S}_r$, as well as $d\mathsf{S}$ is the area element of $\mathsf{S}_r$ in the coordinate chart, and hence
    \begin{equation}
    \label{62}
    \mathsf{M}_{\mathrm{ADM}}(\mathbb{M}^{n},{\mathsf{g}})=0\Longleftrightarrow (\mathbb{M}^{n},{\mathsf{g}})\cong(\mathbb R^n, \mathsf{g}_0).
    \end{equation}

    Obviously, \eqref{61}-\eqref{62} can be well-understood via \cite[Example 2.3]{HL} as shown below.
    \begin{example}\label{e}
    	The typical AF manifold is the Schwarzchild manifold:
    	$$
    	\begin{cases}
    	(\mathbb M^n_{\mathrm S},{\mathsf g}_{\mathrm S})=\bigg(\mathbb R^n\setminus \bar{B}\big({(\frac{\mathsf{m}}{2})^\frac1{n-2}}\big),\ \Big(1+2^{-1}{\mathsf{m}}{|\cdot|^{2-n}}\Big)^\frac{4}{n-2}\mathsf{g}_0(\cdot)\bigg);\\
    	\bar{B}\big({(\frac{\mathsf{m}}{2})^\frac1{n-2}}\big)=\Big\{x\in\mathbb R^n:\ |x|\le (\frac{\mathsf{m}}{2})^\frac1{n-2}\Big\};\\
    			\begin{tikzpicture}[scale=0.466, blue]
    			\draw[line width=0.6, dashed] (0,0) ellipse [x radius=4.99cm,y radius=0.4cm];
    			\node at (0,1) {\ $\big\{|x|=\frac{\mathsf{m}}{2}\big\}$ };
    			\draw[line width=0.6] (5,0)..controls (5,-3) and (6,-5)..(9,-6);
    			\draw[line width=0.6] (-5,0)..controls (-5,-3) and (-6,-5)..(-9,-6);
    			\node at (0,-6) {$|x|\rightarrow\infty$\ \rm{in}\ $\mathbb M^3_{\mathrm S}$};
    			\end{tikzpicture}
    			
    	\end{cases}
    	$$
    	with:
    	\begin{itemize}
    		
    		\item its ADM mass being
    		\begin{equation*}
    		\mathsf{M}_{\mathrm{ADM}}(\mathbb M^n_{\mathrm S},{\mathsf g}_{\mathrm S})=\mathsf{m};
    		\end{equation*}
    		
    		\item its exterior domain outside the horizon being represented as the graph of the spherically symmetric function
    		
    		\begin{equation*}
    		\phi(x)=\begin{cases}
    		\sqrt{8{\mathsf{m}}(|x|-2\mathsf{m})}\ &\ \text{for}\ n=3;\\
    		\sqrt{2\mathsf{m}}\ln\Bigg(\frac{|x|}{\sqrt{2\mathsf{m}}}+\sqrt{\frac{|x|^2}{{2\mathsf{m}}}-1}\Bigg)\ &\ \text{for}\ n=4;\\
    		c_{n,{\mathsf{m}}}+{\it O}(|x|^{2-\frac{n}{2}})\ &\ \text{for}\ n\ge 5;
    		\end{cases}
    		\end{equation*}
    		where $c_{n,{\mathsf{m}}}$ is a constant depending on the pair $\{n,\mathsf{m}\}$.
    	\end{itemize}
    	
    \end{example}

    Moreover, Witten's spinor $\psi$ (cf. \cite[(2.3)]{BO}) produces the following ADM-mass formula
    \begin{equation}
    \label{63}
    \mathsf{M}_{\mathrm{ADM}}(\mathbb{M}^{n},{\mathsf{g}})=\big(2(n-1)\sigma_{n-1}\big)^{-1}\int_{\mathbb M^n}\big(|\nabla(2\psi)|^2+\mathsf{R}_g |\psi|^2\big)\,d\upsilon_{\mathsf{g}}.
    \end{equation}
    According to the argument for \cite[Theorem 2.17]{BO}, there is a strictly positive function $w$ such that not only
    $$
    4\Delta w=\mathsf{R}_g w\ \ \text{with}\ \
    w\to 1\ \ \text{at}\ \ \infty,
    $$
    but also
    \begin{align}
    \label{64}
    \int_{\mathbb M^n}(|\nabla(2w)|^2+\mathsf{R}_g w^2)\,d\upsilon_{\mathsf{g}}&=\underset{{v-1\in C^\infty_0(\mathbb M^n)}}{\inf}\int_{\mathbb M^n}\Big(|\nabla(2v)|^2+\mathsf{R}_g v^2\Big)\,d\upsilon_{\mathsf{g}}\\
    &\le\int_{\mathbb M^n}(|\nabla(2\psi)|^2+\mathsf{R}_g |\psi|^2)\,d\upsilon_{\mathsf{g}}\notag\\
    &=\big(2(n-1)\sigma_{n-1}\big)\mathsf{M}_{\mathrm{ADM}}(\mathbb{M}^{n},{\mathsf{g}})\ \ \text{due to \eqref{63}}.\notag
    \end{align}
  Thus, we are induced by \eqref{64} to wonder whether there is a natural formula for the ADM mass of any $3\le n$-dimensinal AF manifold which does not depend on any spin structure. For this concern, a careful look at \eqref{64}, along with \cite[Theorem 3.3]{Xasma}\&\eqref{11} plus \cite{BM, Es, FS, H, HPR, Sc, Xi3}, suggests such an intrinsic concept of the ADM-type masses of an arbitrary AF manifold that
    \begin{equation}
    \label{65}
    \begin{cases}
   \mathsf{M}_{\mathrm{ADM}_p,\dagger}(\mathbb M^n,\mathsf{g})=\Big({{2(n-1)\sigma_{n-1}}}\Big)^{-1}{\underset{u-1\in C^\infty_0(\mathbb M^n)}{\inf}
    	\int_{\mathbb M^n}\Big(|\nabla(2u)|^p+\mathsf{R}_g |u|^p\Big)\,d\upsilon_{\mathsf{g}}};\\
    \mathsf{M}_{\mathrm{ADM}_p,\ddagger}(\mathbb M^n,\mathsf{g})=
    \Big({{2(n-1)\sigma_{n-1}}}\Big)^{-1}
    {\underset{u-1\in C^\infty_0(\mathbb M^n)}{\inf}
    	\int_{\mathbb M^n}\Big(|\Delta(2u)|^p+\mathsf{R}_g |u|^p\Big)\,d\upsilon_{\mathsf{g}}},
    \end{cases}
    \end{equation}
    where $C^\infty_0(\mathbb M^n)$'s closure
    under the norm
    $$
    \left(\int_{{\mathbb M^n}}\Big(|u|^p+|\nabla u|^p\Big)\,d\upsilon_{\mathsf{g}}\right)^\frac1p\ \ \text{or}\ \
    \left(\int_{{\mathbb M^n}}\Big(|u|^p+|\nabla u|^p+|\nabla^2 u|^p\Big)\,d\upsilon_{\mathsf{g}}\right)^\frac1p
    $$
    is just the Sobolev $[1,\infty)\ni p$-space $$W^{1,p}_{0}(\mathbb M^n)\ \ \text{or}\ \ W^{2,p}_{0}(\mathbb M^n),
    $$
    where $\nabla^k u$ denotes the vector of all partial derivatives of $u$ of order $k\in \{1,2\}$ with respect to $\mathsf{g}$.

     Below is the newly discovered nonnegative $\mathrm{ADM}_p$-mass principle whose \eqref{69} is essentially tied to \eqref{52}-\eqref{53}.

    \begin{theorem}
    	\label{th31} For $p\in (1,\infty)$ let $(\mathbb M^n,\mathsf{g})$ be a $3\le n$-dimensional AF manifold satisfying
    	\begin{equation}
    	\label{66}
    	\begin{cases}
    	\mathsf{R}_g\ge 0;\\
    	 \int_{{\mathbb M^n}}\mathsf{R}_g\,d\upsilon_{\mathsf{g}}<\infty.
    	 \end{cases}
    	\end{equation}
    	If there is a positive constant pair $\{\lambda_\dagger,\lambda_\ddagger\}$ such that
    		\begin{equation}
    			\label{67}
    			\begin{cases}
    		\int_{\mathbb M^n}\Big(|\nabla(2u)|^p+\mathsf{R}_g |u|^p\Big)\,d\upsilon_{\mathsf{g}}\ge \lambda_\dagger\int_{\mathbb M^n}|u{-1}|^p\,d\upsilon_{\mathsf{g}}\ \ \forall\ \ u{-1}\in C_0^\infty(\mathbb M^n);\\
    		\int_{\mathbb M^n}\Big(|\Delta(2u)|^p+\mathsf{R_g} |u|^p\Big)\,d\upsilon_{\mathsf{g}}\ge \lambda_\ddagger\int_{\mathbb M^n}|u{-1}|^p\,d\upsilon_{\mathsf{g}}\ \ \forall\ \ u{-1}\in C_0^\infty(\mathbb M^n),
    		\end{cases}
    		\end{equation}
    		then
    		\eqref{65} has a minimizer pair $\{u_\dagger,u_\ddagger\}$ with
    		$$
    		\begin{cases}
    		{u_\dagger}\in W^{1,p}_{0}(\mathbb M^n)\cap C^{1,\alpha_\dagger}(\mathbb M^n)\ \ \text{with some}\ \  \alpha_\dagger\in (0,1);\\
    		{u_\ddagger}\in W^{2,p}_{0}(\mathbb M^n)\cap C^{2,\alpha_\ddagger}(\mathbb M^n)\ \ \text{with some}\ \  \alpha_\ddagger\in (0,1),
    		\end{cases}
    		$$ such that not only
    		\begin{equation}
    		\label{68}
    		\begin{cases}
    		\begin{cases}
    		2^p\Delta_pu_\dagger=\mathsf{R_g}{|u_\dagger|^{p-2}u_\dagger}\ \ \text{a.e. on}\ \ \mathbb M^n;\\
    		\underset{x\to\infty}{\lim}{u_\dagger}(x)=1,
    		\end{cases}\\
    		\begin{cases}
    		2^p\Delta\Big(|\Delta {u_\ddagger}|^{p-2}\Delta {u_\ddagger}\Big)=\mathsf{R_g}|{u_\ddagger}|^{p-2}u_\ddagger\ \ \text{a.e. on}\ \ \mathbb M^n;\\
    		\underset{x\to\infty}{\lim}{u_\ddagger}(x)=1,
    		\end{cases}
    		\end{cases}
    		\end{equation}
    		but also
    		\begin{equation}\label{69}
    		\begin{cases}
    		\mathsf{M}_{\mathrm{ADM}_p,\dagger}(\mathbb M^n,\mathsf{g})=\big({2(n-1)\sigma_{n-1}}\big)^{-1}{\int_{{\mathbb M^n}}\mathsf{R_g} |u_\dagger|^{p-2}u_\dagger\,d\upsilon_{\mathsf{g}}}=\frac{\underset{r\to\infty}{\lim}\int_{\mathsf{S}_r}|\nabla {u_\dagger}|^{p-2}\langle\nabla {u_\dagger},\nu\rangle\,d\mathsf{S}}{2^{1-p}(n-1)\sigma_{n-1}};\\
    		
 \mathsf{M}_{\mathrm{ADM}_p,\ddagger}(\mathbb M^n,\mathsf{g})=\big({2(n-1)\sigma_{n-1}}\big)^{-1}{\int_{{\mathbb M^n}}\mathsf{R_g} |u_\ddagger|^{p-2}u_\ddagger\,d\upsilon_{\mathsf{g}}}.
    		\end{cases}
    		\end{equation}
    		Consequently, there holds
    		\begin{equation}\label{610}
    		\begin{cases}	
    		
    		\mathsf{M}_{\mathrm{ADM}_p,\dagger}(\mathbb M^n,\mathsf{g})=0\Longleftrightarrow \mathsf{R_g}=\frac{2^p\Delta_p u_\dagger}{|u_\dagger|^{p-2}u_\dagger}=0\ \ \text{a.e. on}\ \ \mathbb M^n;\\
\mathsf{M}_{\mathrm{ADM}_p,\ddagger}(\mathbb M^n,\mathsf{g})=0\Longleftrightarrow \mathsf{R_g}=
    		\frac{2^p\Delta\big(|\Delta {u_\ddagger}|^{p-2}\Delta {u_\ddagger}\big)}{|{u_\ddagger}|^{p-2}u_\ddagger}=0\ \ \text{a.e. on}\ \ \mathbb M^n.
    		\end{cases}
    		\end{equation}
    		
\end{theorem}
    \begin{proof} It suffices to validate the assertion for $\nabla u_\dagger$ since the argument for $\Delta u_\ddagger$ is completely similar.
    	
    	Firstly, note that \eqref{66} guarantees
    	\begin{equation}
    	\label{611}
    	0\le \mathsf{M}_{\mathrm{ADM}_p,\dagger}(\mathbb{M}^{n},{\mathsf{g}})\le
    	\int_{\mathbb M^n}\mathsf{R_g}\,d\upsilon_{\mathsf{g}}<\infty,
    	\end{equation}
    	So, not only keeping \eqref{611}\&\eqref{67}'s $\dagger$-inequality in mind but also following \cite[p.86]{He}, we may assume that
    	$$
    	\text{$u_j-1\in C^\infty_0(\mathbb M^n)$ is a minimizing sequence for $\mathsf{M}_{\mathrm{ADM}_p,\dagger}(\mathbb M^n,\mathsf{g})$}.
    	$$
    	Then
    	$$
    	\mathsf{M}_{\mathrm{ADM}_p,\dagger}(\mathbb M^n,\mathsf{g})=\lim_{j\to\infty}\frac{\int_{\mathbb M^n}(|\nabla(2u_j)|^p+\mathsf{R_g} |u_j|^{p})\,d\upsilon_{\mathsf{g}}}{{2(n-1)\sigma_{n-1}}}.
    	$$
    	Note that $u_j{-1}$ is bounded in $W^{1,p}_0(\mathbb M^n)$. So we use not only the fact that $W^{1,p}_0(\mathbb M^n)$ is reflexive under $p\in (1,\infty)$ but also the well-known Rellich-Kondrakov theorem to get a function
    	$$
    	u_\dagger-1\in W^{1,p}_0(\mathbb M^n)
    	$$
    	such that
    	$$
    	\begin{cases}
    	u_j\rightharpoonup {u_\dagger}\ \ &\text{in}\ \ W^{1,p}_0(\mathbb M^n);\\
    	\|u_j-{u_\dagger}\|_{p}=\left(\int_{\mathbb M^n}|u_j-{u_\dagger}|^p\,d\upsilon_{\mathsf{g}}\right)^\frac1p\to 0;\\
    	u_j\to {u_\dagger}\ \ &\text{a.e. on}\ \ \mathbb M^n;\\
    	\underset{x\to\infty}{\lim}u_\dagger(x)=1\ \ (\text{verifying \eqref{68}'s second formula}).
    	\end{cases}
    	$$
    	Moreover, the above weak convergence derives
    	\begin{align*}
    	\mathsf{M}_{\mathrm{ADM}_p,\dagger}(\mathbb M^n,\mathsf{g})&\le\frac{ \int_{\mathbb M^n}\Big(|\nabla(2{u_\dagger})|^p+\mathsf{R_g} |{u_\dagger}|^p\Big)\,d\upsilon_{\mathsf{g}}}{2(n-1)\sigma_{n-1}}\\
    	&\le\frac{\underset{j\to\infty}{\liminf}\int_{\mathbb M^n}\Big(|\nabla(2u_j)|^p+\mathsf{R_g} |u_j|^p\Big)\,d\upsilon_{\mathsf{g}}}{2(n-1)\sigma_{n-1}}\\
    	&\le\mathsf{M}_{\mathrm{ADM}_p,\dagger}(\mathbb M^n,\mathsf{g}).
    	\end{align*}
    	Thus, we have
    	\begin{equation}
    	\label{612}
    	\mathsf{M}_{\mathrm{ADM}_p,\dagger}(\mathbb M^n,\mathsf{g})=\big({2(n-1)\sigma_{n-1}}\big)^{-1}
    	{\int_{\mathbb M^n}\Big(|\nabla(2{u_\dagger})|^p+\mathsf{R_g} |u_\dagger|^p\Big)\,d\upsilon_{\mathsf{g}}}.
    	\end{equation}
    	The formula \eqref{612} in turn implies that ${u_\dagger}$ is a critical point of the energy functional
    	$$
    	W_0^{1,p}(\mathbb M^n)\ni v-1\mapsto	\mathsf{E}(v)=\int_{\mathbb M^n}\Big(|\nabla (2 v)|^p+\mathsf{R_g}|v|^p\Big)\,d\upsilon_{\mathsf{g}},
    	$$
    	whence ${u_\dagger}$ enjoys
    	\begin{align*}
    	\partial_t \mathsf{E}\Big({u_\dagger}+t\phi\Big)\bigg|_{t=0}
    	&=\int_{\mathbb M^n}\Big(2^p|\nabla {u_\dagger}|^{p-2}\langle\nabla {u_\dagger},\nabla\phi\rangle+\mathsf{R_g}{|u_\dagger|^{p-2}u_\dagger}\phi\Big)\,d\upsilon_{\mathsf{g}}\\
    	&= 0\ \ \forall\ \ \phi\in C_0^\infty(\mathbb M^n).
    	\end{align*}
    	Furthermore, an integration-by-parts produces that if $\phi\in C_0^\infty(\mathbb M^n)$ then
    	$$
    	\int_{\mathbb M^n}|\nabla {u_\dagger}|^{p-2}\langle\nabla {u_\dagger},\nabla\phi\rangle\,d\upsilon_{\mathsf{g}}
    	=-\int_{\mathbb M^n}\phi\langle\nabla,|\nabla {u_\dagger}{|}^{p-2}\nabla {u_\dagger}\rangle\,d\upsilon_{\mathsf{g}},
    	$$
    	whence
    	$$
    	\int_{\mathbb M^n}\Big(\mathsf{R_g}{|u_\dagger|^{p-2}u_\dagger}-2^p\Delta_pu_\dagger\Big)\phi\,d\upsilon_{\mathsf{g}}=0.
    	$$
    	This validates the desired $\dagger$-equation in \eqref{68} which ensures $u_\dagger\in C^{1,\alpha_\dagger}(\mathbb M^n)$ for some $\alpha_\dagger\in (0,1)$ via a standard regularity theory for the elliptic equations.
    	
    	Secondly, we utilize not only the first equation of \eqref{68} but also the integration-by-parts as well as ${u_\dagger}\to 1$ at $\infty$ to conclude that if $\mathbb B_r$ is sufficiently large coordinate ball in $\mathbb M^n$ with its boundary $\mathbb S_r$ then a combination of \eqref{68}'s first $\dagger$-condition and the integration-by-parts derives
    	\begin{align*}
    	&\int_{\mathbb M^n}\Big(|\nabla(2{u_\dagger})|^p+\mathsf{R_g} |u_\dagger|^p\Big)\,d\upsilon_{\mathsf{g}}\\
    	&\ \ = \int_{\mathbb M^n}\Big(|\nabla(2{u_\dagger})|^p + 2^p u_\dagger\Delta_pu_\dagger\Big)\,d\upsilon_{\mathsf{g}}\\
    	&\ \ =2^p\underset{r\to\infty}{\lim}\bigg(\int_{\mathbb M^n\setminus\mathbb{B}_r}+\int_{\mathbb B_r}\bigg)\Big(|\nabla {u_\dagger}|^p+{u_\dagger}\Delta_pu_\dagger\Big)\,d\upsilon_{\mathsf{g}}\\
    	&\ \ =2^p\underset{r\to\infty}{\lim}\int_{\mathbb B_r}\Big(|\nabla {u_\dagger}|^{p-2}\langle\nabla {u_\dagger},\nabla {u_\dagger}\rangle+{u_\dagger}\Delta_pu_\dagger\Big)\,d\upsilon_{\mathsf{g}}\\ &\ \ =2^p\underset{r\to\infty}{\lim}\int_{\mathsf{S}_r}{u_\dagger}|\nabla {u_\dagger}|^{p-2}\langle\nabla {u_\dagger},\nu\rangle\,d\mathsf{S}\\
    	&\ \ =2^p\underset{r\to\infty}{\lim}\int_{\mathsf{S}_r}|\nabla {u_\dagger}|^{p-2}\langle\nabla {u_\dagger},\nu\rangle\,d\mathsf{S}\ \ \text{due to}\ \ \underset{x\to\infty}{\lim}u_\dagger(x)=1,
    	\end{align*}
    	thereby yielding the desired $\dagger$-formula in \eqref{69} via not only the divergence theorem but also \eqref{68}'s first $\dagger$-equation as well as \eqref{612}.
    	
    	Thirdly, \eqref{610}'s first $\dagger$-equivalence follows immediately from  resolving
    	\begin{align*}
    	0&=\mathsf{M}_{\mathrm{ADM}_p,\dagger}(\mathbb M^n,\mathsf{g})\\
    	&={\int_{\mathbb M^n}\Bigg(\frac{|\nabla u_\dagger|^p+\mathsf{R_g} |2^{-1}u_\dagger|^p}{2^{1-p}(n-1)\sigma_{n-1}}\Bigg)\,d\upsilon_{\mathsf{g}}}\\
    	&=\int_{{\mathbb M^n}}\left(\frac{\mathsf{R_g} |u_\dagger|^{p-2}u_\dagger}{2(n-1)\sigma_{n-1}}\right)\,d\upsilon_{\mathsf{g}}\\
    	&={\underset{r\to\infty}{\lim}\int_{\mathsf{S}_r}\Bigg(\frac{|\nabla {u_\dagger}|^{p-1}\Big\langle\frac{\nabla {u_\dagger}}{|\nabla u_\dagger|},\nu\Big\rangle}{2^{1-p}(n-1)\sigma_{n-1}}\Bigg)\,d\mathsf{S}}.
    	\end{align*}
    	
    \end{proof}	

    \begin{remark}\label{r31} Clearly, the elementary formula
    	$$
    	\underset{p\to1}{\lim}(a^p+b^p)=a+b\ \ \forall\ \ (a,b)\in[0,\infty)^2,
    	$$
    	along with Theorem \ref{th31}, implies
    		$$
    		\begin{cases}
    		\underset{p\to 1}{\lim}\mathsf{M}_{\mathrm{ADM}_p,\dagger}(\mathbb M^n,\mathsf{g})={\underset{u-1\in C^\infty_0(\mathbb M^n)}{\inf}\int_{\mathbb M^n}\Big(\frac{|\nabla(2u)|+\mathsf{R_g} |u|}{2(n-1)\sigma_{n-1}}\Big)\,d\upsilon_{\mathsf{g}}};\\
    	\underset{p\to 1}{\lim}\mathsf{M}_{\mathrm{ADM}_p,\ddagger}(\mathbb M^n,\mathsf{g})={\underset{u-1\in C^\infty_0(\mathbb M^n)}{\inf}\int_{\mathbb M^n}\Big(\frac{|\Delta(2u)|+\mathsf{R_g} |u|}{2(n-1)\sigma_{n-1}}\Big)\,d\upsilon_{\mathsf{g}}}.
    	\end{cases}
    	$$
    	Interestingly, in accordance with \cite[p.1354]{BBCO}, we can naturally view
    	$$
    	\mathsf{M}_{\mathrm{ADM}_{2,\dagger}}(\mathbb R^3,\mathsf{g}_0)
    	$$
    	as the scattering length of the scalar curvature
    	$$
    	\mathsf{R_g}\ \ \text{of the Euclidean $3$-space $(\mathbb R^3,\mathsf{g}_0)$}.
    	$$	
    \end{remark}

\end{document}